\documentclass[10pt,leqno]{article}
\usepackage{amsmath,amsfonts,amssymb,amscd}
\usepackage[a4paper,width=13cm, height=21cm]{geometry}

\usepackage{theorem}
\usepackage{url}
\usepackage{hyperref}
\usepackage{tikz-cd}
\usetikzlibrary{cd}

\theoremstyle{change}
\theorembodyfont{\sl}
\newtheorem{Theorem}[subsection]{Theorem}
\newtheorem{Proposition}[subsection]{Proposition}
\newtheorem{Lemma}[subsection]{Lemma}

\theorembodyfont{\rmfamily}

\newtheorem{Remark}[subsection]{Remark}

\numberwithin{equation}{subsection}

\newenvironment{prf}[1]{\trivlist
\item[\hskip \labelsep{\it
#1\hspace*{.3em}}]}{~\hspace{\fill}~$\square$\endtrivlist}

\newenvironment{proof}{\begin{prf}{\bf Proof}}{\end{prf}}

\newcommand{\ol}{\overline}
\newcommand{\pr}{\mathrm{pr}}
\newcommand{\add}{\mathrm{add}}
\newcommand{\C}{\mathbb C}
\newcommand{\R}{\mathbb R}
\newcommand{\Q}{\mathbb Q}
\newcommand{\Z}{\mathbb Z}
\newcommand{\F}{\mathbb F}

\newcommand{\G}{\mathbb G}
\newcommand{\Gm}{\G_{\mathrm{m}}}
\newcommand{\Ga}{\G_{\mathrm{a}}}
\renewcommand{\H}{\mathbb H}

\newcommand{\Pic}{\mathrm{Pic}}
\newcommand{\calO}{\mathcal{O}}

\newcommand{\calL}{\mathcal{L}}
\newcommand{\calM}{\mathcal{M}}
\newcommand{\calN}{\mathcal{N}}
\newcommand{\calP}{\mathcal{P}}
\newcommand{\divisor}{\mathrm{div}}
\newcommand{\Div}{\mathrm{Div}}
\newcommand{\quot}{\mathrm{quot}}
\newcommand{\End}{\mathrm{End}}
\newcommand{\Isom}{\mathrm{Isom}}
\newcommand{\Norm}{\mathrm{Norm}}
\newcommand{\Ext}{\mathrm{Ext}}
\newcommand{\Hom}{\mathrm{Hom}}
\newcommand{\GL}{\mathrm{GL}}
\newcommand{\GSp}{\mathrm{GSp}}
\newcommand{\Sp}{\mathrm{Sp}}

\newcommand{\id}{\mathrm{id}}
\newcommand{\lto}{\longrightarrow}

\newcommand{\diag}{\mathrm{diag}}
\newcommand{\Gr}{\mathrm{Gr}}

\renewcommand{\Im}{\mathrm{Im}}
\renewcommand{\Re}{\mathrm{Re}}
\newcommand{\rmH}{\mathrm{H}}
\newcommand{\rmM}{\mathrm{M}}
\newcommand{\tors}{\mathrm{tors}}
\newcommand{\transl}{\mathrm{tr}}
\renewcommand{\phi}{\varphi}
\newcommand{\Sh}{\mathrm{Sh}}

\begin{document}

\title{Pink's conjecture on unlikely intersections and families of
  semi-abelian varieties \footnote{AMS Classification: 14K05, 14G35,
    11G15, 14K30, 11G15. Key words: semi-abelian varieties; Poincar\'e
    biextensions; mixed Shimura varieties; Manin-Mumford, Andr\'e-Oort
    and Zilber-Pink conjectures.}}
\author{ Daniel Bertrand  \&  Bas  Edixhoven \\
\small{
  \href{mailto:daniel.bertrand@imj-prg.fr}{daniel.bertrand@imj-prg.fr} \quad
  \href{mailto:edix@math.leidenuniv.nl}{edix@math.leidenuniv.nl}}
}

\maketitle

\begin{abstract}

  The Poincar\'e torsor of a Shimura family of abelian varieties can
  be viewed both as a family of semi-abelian varieties and as a mixed
  Shimura variety. We show that the special subvarieties of the latter
  cannot all be described in terms of the subgroup schemes of the
  former. This provides a counter-example to the relative
  Manin-Mumford conjecture, but also some evidence in favour of Pink's
  conjecture on unlikely intersections in mixed Shimura varieties. The
  main part of the article concerns mixed Hodge structures and the
  uniformisation of the Poincar\'e torsor, but other, more geometric,
  approaches are also discussed.

\medskip
  
  \centerline{\textbf{R\'esum\'e}}
  
  \medskip

  (Sur la conjecture de Pink sur les intersections exceptionnelles et
  les familles de vari\'et\'es semi-ab\'eliennes) Le torseur de
  Poincar\'e d'une famille de Shimura de vari\'et\'es ab\'eliennes
  s'interpr\`ete \`a la fois comme une famille de vari\'et\'es
  semi-ab\'eliennes et comme une vari\'et\'e de Shimura mixte. Nous
  montrons que ses sous-vari\'et\'es sp\'eciales en ce deuxi\`eme sens
  ne peuvent pas toutes se d\'ecrire en termes de sous-sch\'emas en
  groupes. Cela donne un contre-exemple \`a la conjecture de
  Manin-Mumford relative, mais t\'emoigne aussi de la pertinence de la
  conjecture de Pink sur les intersections exceptionnelles dans les
  vari\'et\'es de Shimura mixtes. L'essentiel de l'article porte sur
  les structures de Hodge mixtes, mais d'autres approches, de nature
  plus g\'eom\'etrique, sont aussi abord\'ees.

\end{abstract}

\setcounter{tocdepth}{1}
\tableofcontents

\section{Introduction}
In the unpublished preprint~\cite{Pink2} Pink formulated a very
influential conjecture (the equivalent Conjectures~1.1--1.3) on
so-called ``unlikely intersections'' in mixed Shimura varieties. 
Here we merely recall the statement of his Conjecture~1.3: 
\begin{quote}
  if $Y$ is a Hodge generic irreducible closed subvariety of a mixed
  Shimura variety~$S$, then the union of the intersections of $Y$ with
  the special subvarieties of $S$ of codimension at least $\dim(Y)+1$
  is not Zariski dense in~$Y$.
\end{quote}
We refer to~\cite{Zannier} for more details on such intersections, and
for their relations to the conjectures by Manin--Mumford,
Mordell--Lang (which are now theorems), and Andr\'e--Oort.  See also
\cite{Pink2}, \cite{Pink1}, and~\cite{KUY}.  The Andr\'e--Oort
conjecture was recently proved for all $\mathcal{A}_g$
in~\cite{Tsimerman}.

In the last section of~\cite{Pink2}, Pink states a relative version of
the Manin-Mumford conjecture for families of semi-abelian varieties,
Conjecture~6.1:
\begin{quote}
  if $B\to X$ is a family of semi-abelian varieties over~$\C$ and $Y$
  is an irreducible closed subvariety in $B$ that is not contained in
  any proper closed subgroup scheme of $B\to X$, then the union of the
  intersections of $Y$ with algebraic subgroups of
  codimension at least $\dim(Y)+1$ of the fibres of $B\to X$ is not
  Zariski dense in~$Y$.
\end{quote}
Furthermore, Theorem~6.3 of~\cite{Pink2} claims that Conjecture~1.3
implies Conjecture~6.1. However, a counter-example to Conjecture~6.1
was given in the unpublished preprint~\cite{Bertrand11}, based on a
relative version of a construction of Ribet
(\cite{Jacquinot-Ribet},~\cite{Ribet}), leading to the notion of Ribet
sections on certain semi-abelian schemes. But it was also shown
in~\cite{Bertrand11} that this counter-example was not in
contradiction with Conjecture~1.3, and so, the error was in the proof
of Theorem~6.3 (see
Remark~\ref{rem_r_f_msv}.\ref{rem_r_f_msv_item_4} at the end of
Section~\ref{sec5} below). The conclusion is
that the context of mixed Hodge structures \emph{is} the right one for
a relative Manin-Mumford conjecture for families of
\emph{semi}-abelian varieties: indeed, the image of a Ribet section is
a special subvariety that can in general \emph{not} be interpreted as
a subgroup scheme (see
Remark~\ref{rem_r_f_msv}.\ref{rem_r_f_msv_item_2} below). However, for
families of \emph{abelian} varieties (that is, mixed Shimura varieties
of Kuga type), Theorem~6.3 is correct, see~\cite{Pink1},
Proposition~4.6, \cite{Gao-Pisa}, Proposition~3.4, and again
Remark~\ref{rem_r_f_msv}.\ref{rem_r_f_msv_item_4}.

The aim of this article is to provide not only a published account of
this story, sharpening the results of~\cite{Bertrand11}, but also a
self-contained description of the involved mixed Hodge structures and
the corresponding mixed Shimura varieties, made as accessible as
possible.

The article is structured as follows. In Section~\ref{sec2} we present
the (counter)example, in the case of complex elliptic curves with
complex multiplications, and in Section~\ref{sec3} (which introduces a
different viewpoint) for abelian schemes. In Sections~\ref{sec4} and
\ref{sec5} we give the description of the example in the context of
mixed Shimura varieties whose pure part parametrises principally
polarised abelian varieties.  We show that it gives evidence in favour of
Pink's Conjecture~1.3. Finally, in Section~\ref{sec6} we give a description
of the example, in the case of elliptic curves, in terms of
generalised jacobians.

\begin{Remark}
  In each section, we construct Ribet sections under various
  denominations, namely $t_\phi$ in~\eqref{eq:def_t_phi}, $r_f$ in
  Proposition~\ref{prop:ribsec}, $r^\Sh_f$ in Thm~\ref{thm:RibetsectionMSV},
  and $t^J_\phi$ in~\eqref{eq:def_sJ}.  At each step, we prove their
  compatibility, as well as some of their properties. The main
  property, leading to the searched-for counterexample to
  Conjecture~6.1 of~\cite{Pink2}, is stated in
  Theorem~\ref{thm:ttorsionproperty} and asserts that the Ribet section
  $t_\phi$ maps torsion points of the base to torsion points of their
  fibres.  The proof (with sharper additional properties) is given in
  terms of $r_f$ in Proposition~\ref{prop:ribsectors}, of $r^\Sh_f$ in
  Proposition~\ref{prop_tors_prop_r_f_msv} and of
  $t^J_\phi$ in Theorem~\ref{thm:tJ}.  So, these proofs have logically
  unnecessary overlaps, but their settings are sufficiently distinct
  to justify this presentation. We should mention that yet another
  construction of the Ribet sections was proposed
  in~\cite{Bertrand11}, based as in~\cite{Jacquinot-Ribet} on the
  theory of 1-motives. But as shown in~\cite{Chambert-Loir99}, the latter
  is equivalent to the construction of $t_\phi$ in Section~\ref{sec2}.
\end{Remark}
\begin{Remark}
  We will sometimes abbreviate ``the image of a given section'' by
  ``the section''. On the other hand, the image of a Ribet section
  will be called a Ribet variety.
\end{Remark}
\begin{Remark}
  One may wonder if, in spite of the above mentioned error in
  Theorem~6.3 of~\cite{Pink2}, Pink's general Conjecture 1.3 can still
  be applied to the study of unlikely intersections in semi-abelian
  varieties. Bertrand, who could see this only under strong
  assumptions of simplicity (and only for Manin-Mumford), suggested
  that Edixhoven study the problem in full generality. And indeed,
  after this article was finished, Edixhoven found that everything in
  Sections~4 and~5 of~\cite{Pink2} is correct, except the proof of the
  last statement, Theorem~5.7. That theorem states that Conjecture~1.3
  implies Conjecture~5.1, the unlikely intersection variant of the
  Manin-Mumford conjecture for semi-abelian varieties. Moreover, he
  also showed that, with a small change, and a more detailed
  description of the special subvarieties of the mixed Shimura
  varieties involved, Pink's argument gives that Conjecture~1.3
  implies Conjecture~5.2 (unlikely intersection generalisation of
  Mordell-Lang), and therefore, by Theorem~5.5 of~\cite{Pink2},
  implies Conjecture~5.1. The details of this will appear in an
  article in preparation by Edixhoven.
\end{Remark}

\section{The example with elliptic curves}\label{sec2}
The key player in the example in~\cite{Bertrand11} is the Poincar\'e
torsor $\calP$ on a product $E\times E^\vee$, where $E$ is a complex
elliptic curve and where $E^\vee$ is its dual. 

To make $\calP$ and $E^\vee$ more explicit, we use the isomorphism
$\lambda\colon E \to E^\vee$ that sends a point $P$ to the class of
the invertible $\calO$-module $\calO((-P)-0)$, isomorphic to
$\calO(0-P)$ (this is the unique principal polarisation of~$E$). In
the notation of~\cite[Section~6]{MumfordAV}, $\lambda=\phi_\calM$,
where $\calM$ is the invertible $\calO$-module $\calO(0)$ on~$E$, and
where $\phi_\calM$ sends $P$ to the class of
$(\transl_P^*\calM)\otimes_\calO\calM^{-1}$, with $\transl_P$ the
translation by $P$ map on~$E$.

The \emph{Poincar\'e bundle} $\calL$ on $E\times E$ is then
\begin{equation}\label{eqn:defPbundle}
\calL = 
\add^*\calM\otimes_\calO\pr_1^*\calM^{-1}\otimes_\calO\pr_2^*\calM^{-1}\otimes_\calO0^*\calM,  
\end{equation}
where $\add$, $\pr_1$, $\pr_2$, and~$0$ are the addition map, the
projections, and the constant map~$0$ from $E\times E$ to~$E$. It is
isomorphic (with the isomorphism given by the choice of a non-zero
element of the fibre $\calM(0)$ of~$\calM$ at~$0$, i.e., of a non-zero
tangent vector of $E$ at~$0$) to $\calO(D)$, with
\begin{equation}\label{eqn:defD}
D=\add^{-1}0-\pr_1^{-1}0-\pr_2^{-1}0
\end{equation}
The fibre $\calL(x,y)$ at a point $(x,y)$ is given by:
\begin{equation}\label{eqn:fibresPbundle}
\calL(x,y) =
\calM(x+y)\otimes\calM(x)^{-1}\otimes\calM(y)^{-1}\otimes\calM(0).   
\end{equation}
In particular:
$\calL(x,0)=\calM(x)\otimes\calM(x)^{-1}\otimes\calM(0)^{-1}\otimes\calM(0)=\C$, 
and similarly for $\calL(0,y)$. Hence $\calL$ is canonically trivial on the union of
$E\times\{0\}$ and $\{0\}\times E$. But let us remark that the
pullback of $\calL$ via $\diag\colon E\to E\times E$ has fibre at $x$
equal to $\calM(2x)\otimes\calM(x)^{-2}\otimes\calM(0)$, hence is
given by the divisor $\sum_{P\in E[2]}P - 2{\cdot}0$ which is of degree
2 and linearly equivalent to~$2{\cdot}0$.

The \emph{Poincar\'e torsor} $\calP$ is then the $\Gm$-torsor on
$E\times E$ (trivial locally for the Zariski topology) of isomorphisms
from $\calO$ to~$\calL$:
\begin{equation}\label{eqn:defPtorsor}
\calP = \Isom(\calO,\calL).  
\end{equation}
It is represented by a complex algebraic variety over $E\times E$,
also denoted~$\calP$. Its fibre $\calP(x,y)$ over $(x,y)$ is
the $\C^\times$-torsor $\Isom(\C,\calL(x,y))$. 

The theorem of the cube (\cite[Section~6]{MumfordAV}) says that any
invertible $\calO$-module $\calN$ on $E^n$ with $n\geq3$, whose
restrictions to $\ker(\pr_i)$ are trivial for all $i$ in
$\{1,\ldots,n\}$, is trivial. For every such~$\calN$, for any non-zero
element $s_0$ of $\calN(0,\ldots,0)$ there is a unique $s$ in
$\calN(E^n)$ such that $s(0)=s_0$ (the reason is that
$\calO(E^n)=\C$).

For example, the invertible $\calO$-module
\[
\bigotimes_{I\subset\{1,2,3\}}\add_I^*\calM^{(-1)^{\# I}} \quad \text{on $E\times E\times E$},
\]
where $\add_I\colon E^3\to E$, $(x_1,x_2,x_3)\mapsto\sum_{i\in I}x_i$, is canonically
trivial (canonically because its fibre at $(0,0,0)$ is
$\calM(0)^{\otimes 4}\otimes\calM(0)^{\otimes-4}=\C$). 
Explicitly: for all points $(x,y,z)$ of $E^3$ we have 
\[
\begin{aligned}
\calM(x+y+z)& \otimes\calM(x+y)^{-1}\otimes\calM(x+z)^{-1}\otimes\calM(y+z)^{-1}\\
& \otimes\calM(x)\otimes\calM(y)\otimes\calM(z) \otimes\calM(0)^{-1}=\C.    
\end{aligned}
\]
Similarly, the invertible $\calO$-modules on $E^3$ with fibres 
\[
\calL(x,y+z)\otimes\calL(x,y)^{-1}\otimes\calL(x,z)^{-1}
\quad\text{and}\quad 
\calL(x+y,z)\otimes\calL(x,z)^{-1}\otimes\calL(y,z)^{-1}
\]
are canonically trivial. Therefore, for all points $x$, $y$ and $z$ of
$E$ we have:
\begin{equation}\label{eqn:L=biadditive}
\calL(x,y+z) = \calL(x,y)\otimes\calL(x,z), 
\quad \calL(x+y,z)=\calL(x,z)\otimes\calL(y,z).
\end{equation}
This gives two composition laws on $\calP$: for
$\alpha\colon\C\to\calL(x,y)$ in $\calP(x,y)$ and
$\beta\colon\C\to\calL(x,z)$ in $\calP(x,z)$ we get
$\alpha\otimes\beta\colon\C\to\calL(x,y+z)$ in $\calP(x,y+z)$, and
similarly with the 2nd variable fixed. With the first variable fixed,
$\calP$ is a commutative group-variety over $E$, via~$\pr_1$, whose
fibres are extensions of $E$ by~$\Gm$, and similarly for~$\pr_2$; for
details, see Chapter~I, Section~2.5 of~\cite{Moret-Bailly} and the Proposition
of Section~2.6 there. In particular, $\calP$ is a \emph{bi-extension} of $E$
and $E$ by~$\Gm$: the two partial group laws commute with each other
in the following sense. For $x_1$, $x_2$, $y_1$ and $y_2$ in $E$, and
$p_{i,j}$ in $\calP(x_i,y_j)$, the various ways of summing the
$p_{i,j}$ leads to the same result in $\calP(x_1+x_2,y_1+y_2)$. This
is proved by considering the universal case $T:=E^4$, $x_1=\pr_1$,
$x_2=\pr_2$, $y_1=\pr_3$ ad $y_2=\pr_4$, and concluding that the
trivialisations of
\[
\calL(x_1+x_2,y_1+y_2)\otimes\calL(x_1,y_1)^{-1}\otimes\calL(x_1,y_2)^{-1}
\otimes\calL(x_2,y_1)^{-1}\otimes\calL(x_2,y_2)^{-1}
\]
corresponding to the various ways of summing are equal because they
are so at $(0,0,0,0)$: writing it out in terms of $\calM$ leads to the
tensor product of as many $\calM(0)$'s as $\calM(0)^{-1}$'s.

With these preliminaries behind us, we can finally proceed to the
construction of Ribet sections. Let $\phi$ be an endomorphism of~$E$ 
and let $\ol{\phi}:=\lambda^{-1}\circ\phi^\vee\circ\lambda$ be the
conjugate of~$\phi$. Let
\[
\gamma=
(\id,\phi-\ol{\phi})\colon 
E\to E\times E, \quad P\mapsto (P,(\phi-\ol{\phi})(P))
\]
be the graph map attached to $\phi-\ol{\phi}$. The following fact was
observed in~\cite{Breen}; see also~\cite{Jacquinot-Ribet} for
a description in terms of 1-motives.
\begin{Proposition}\label{prop:t_phi}
The invertible $\calO$-module $\gamma^*\calL$ on $E$ is canonically trivial.
\end{Proposition}
\begin{proof}
  As this is the crucial ingredient of the example that we present in
  this article, we give two proofs: one for readers who prefer a
  computation using divisors, and one for those who prefer universal
  properties. But first we note that if $\phi=\ol{\phi}$, then
  $\gamma=(\id,0)$ and $\gamma^*\calL$ is canonically trivial because,
  as mentioned above, $\calL$ is canonically trivial on
  $E\times\{0\}$. So in the first proof below we may and do assume that
  $\phi\neq\ol{\phi}$. 

\subsubsection*{A proof by divisors.} The fibre of
$\gamma^*\calL$ at~$0$ is $\calL(0,0)=\C$, and $\calL$ is
isomorphic to $\calO(D)$ with
\[
D=\add^{-1}0-\pr_1^{-1}0-\pr_2^{-1}0
\]
as in~(\ref{eqn:defD}). So it suffices to show that $\gamma^*D$ is
linearly equivalent to the zero divisor on~$E$. Let
$\alpha:=\phi-\ol{\phi}$. We note that
\[
\add\circ\gamma = \add\circ(\id,\alpha) = \id+\alpha,\quad
\pr_1\circ\gamma = \id,\quad\text{and}\quad
\pr_2\circ\gamma = \alpha.
\]
Hence we have the following equalities of divisors on~$E$:
\[
(\id,\alpha)^*D  =  (\id+\alpha)^*0 - \id^*0 - \alpha^*0 
= \sum_{P\in\ker(\id+\alpha)}P - 0 - \sum_{Q\in\ker(\alpha)}Q.
\]
The degree of this divisor is zero because, in $\End(E)$, $\alpha$ is
imaginary, so we have 
\[
\deg(\id+\alpha) =
(\id+\alpha)(\id+\ol{\alpha})=\id+\alpha\ol{\alpha}
=1+\deg(\alpha).
\]
Any degree zero divisor on $E$ is linearly equivalent to $R-0$, with
$R$ the image of the divisor under the group morphism $\Div^0(E) \to
E$ that sends each point to itself. So in our case $R$ is the sum of
the points in $\ker(\id+\alpha)$, minus the sum of the points
in~$\ker(\alpha)$. These two kernels are finite commutative
groups. For such a group, the sum of the elements is $0$, except when
its 2-primary part is cyclic and non-trivial, in which case it is the
element of order~$2$. Let $a:=\phi+\ol{\phi}$ be the trace
of~$\phi$; it is in the subring $\Z$ of~$\End(E)$. Then
$\alpha=-a+2\phi$, and $\id+\alpha=(1-a)+2\phi$. So one of these has
odd degree, and the other is divisible by $2$ in $\End(E)$, and so
for none of them the $2$-primary part of the kernel is cyclic and
non-trivial.

\subsubsection*{A proof by universal properties.} We view 
$E\times E$ as an $E$-scheme via~$\pr_2$. Then $\calL$ is the
universal invertible $\calO$-module of degree~$0$ on~$E$ with given
trivialisation at~$0$: for every complex algebraic variety~$S$ and
every invertible $\calO$-module $\calN$ on $E_S$, fibrewise of
degree~$0$, and with a given trivialisation $\calO_S\to 0^*\calN$,
there is a unique $f\colon S\to E$ such that the pullback of $\calL$
via $\id\times f\colon E_S\to E_E$ is isomorphic to~$\calN$. Moreover,
in this case there is a unique isomorphism $g\colon \calN\to
(\id\times f)^*\calL$ that is compatible with the given
trivialisations at~$0$. Of course, the analogous statements are true
with $\pr_2$ replaced by~$\pr_1$.

Let us turn to~$\ol{\phi}$. It is defined as
$\lambda^{-1}\circ\phi^\vee\circ\lambda$. Hence, for $y$ in $E$,
$\ol{\phi}(y)$ is obtained as follows: $\lambda(y)$ is the isomorphism
class of the invertible $\calO$-module $\calL|_{E\times\{y\}}$ on~$E$,
and then
$\lambda(\ol{\phi}(y))=
(\lambda\circ\lambda^{-1}\circ\phi^\vee\circ\lambda)y
=\phi^\vee(\lambda(y))$ corresponds (by the definition of~$\phi^\vee$)
to~$\phi^*(\calL|_{E\times\{y\}})$.  By definition of~$\lambda$
and~$\calL$, $\lambda(\ol{\phi}(y))$ corresponds
to~$\calL|_{E\times\{\ol{\phi}(y)\}}$. Hence the invertible
$\calO$-modules $(\phi\times\id)^*\calL$ and
$(\id\times\ol{\phi})^*\calL$ on $E\times E$, both trivialised on
$\{0\}\times E$, are uniquely isomorphic on the fibres of the
second projection. Hence we have a canonical isomorphism
between $(\id\times\ol{\phi})^*\calL$ and $(\phi\times\id)^*\calL$.

As $\calL$ together with its trivialisations on $E\times\{0\}$ and
$\{0\}\times E$ is symmetric (that is, invariant under the
automorphism of $E\times E$ that sends $(x,y)$ to $(y,x)$), we get a
canonical isomorphism between $(\id\times\ol{\phi})^*\calL$ and
$(\id\times \phi)^*\calL$.

From (\ref{eqn:L=biadditive}), applied with $x=\id_E$, $y=\phi$ and
$z=-\ol{\phi}$ we get a canonical isomorphism, on~$E$, from
$\gamma^*\calL$ to
$(\id,\phi)^*\calL\otimes(\id,-\ol{\phi})^*\calL$. Applying it again,
but now with $x=\id_E$, $y=\ol{\phi}$ and $z=-\ol{\phi}$, we get a
canonical isomorphism from $\calO$ to
$(\id,-\ol{\phi})^*\calL\otimes(\id,\ol{\phi})^*\calL$, giving us a
canonical isomorphism from $(\id,-\ol{\phi})^*\calL$ to
$(\id,\ol{\phi})^*\calL^{-1}$. Combining, we see that
\[
\begin{aligned}
\gamma^*\calL & = (\id,\phi)^*\calL\otimes(\id,-\ol{\phi})^*\calL
= (\id,\phi)^*\calL\otimes(\id,\ol{\phi})^*\calL^{-1} \\
& = (\id,\phi)^*\calL\otimes(\id,\phi)^*\calL^{-1} = \calO.    
\end{aligned}
\]
\end{proof}
Now we view $\calP$ as a group variety over $E$ via $\pr_1\colon
E\times E\to E$.
The canonical trivialisation
\begin{equation}\label{eq:def_t_phi}
t_\phi\colon\calO\to\gamma^*\calL = (\id,\alpha)^*\calL  
\end{equation}
on $E$ gives, for every $x$ in $E$, an element $t_\phi(x)$ in
$\Isom(\C,\calL(x,\alpha(x)))$, hence an element in
$\calP(x,\alpha(x))$. As such, $t_\phi$ is a section of the group variety
$\calP$ over~$E$, which we call the \emph{Ribet section attached
to~$\phi$}.  

Following~\cite{Bertrand11}, we will now show that if
$\ol{\phi}\neq\phi$, then $t_\phi$ gives a
counterexample to Conjecture~6.1 of~\cite{Pink2}.

\begin{Lemma}\label{lemma:fewsubgroups}
Let $\Gm\rightarrowtail G \twoheadrightarrow E$ be an extension whose
class in the group $\Ext(E,\Gm)$ is not torsion. Then the only connected
algebraic subgroups of $G$ are $\{0\}$, $\Gm$ and~$G$.
\end{Lemma}
\begin{proof}
  Let $H$ be a connected algebraic subgroup of~$G$. Then $\dim(H)$ is
  $0$, $1$ or~$2$. If it is $0$ then $H=\{0\}$, and if it is $2$ then
  $H=G$, so we assume it is~$1$, and that $H$ is not equal
  to~$\Gm$. Then $H\to E$ is surjective, and since $\G_m \cap H$ is a
  finite group, $H\to E$ is an unramified cover. As $H$ is connected,
  it is itself an elliptic curve, and there is an $n\in\Z_{>0}$ and a
  factorisation $n{\cdot}\colon E\to H\to E$. This means that the
  extension $\Gm\rightarrowtail G\twoheadrightarrow E$ is split after
  pullback via $n{\cdot}\colon E\to E$, hence its class is torsion.
\end{proof}

\begin{Lemma}\label{lemma:tgeneric}
  If $\phi\neq\ol{\phi}$, then the union over all $n\in\Z$ of the
  images $(n{\cdot}t_\phi)(E)$ of the sections $n{\cdot}t_\phi$ is
  Zariski dense in~$\calP$.
\end{Lemma}
\begin{proof}
  Let $Z$ be the Zariski closure of the union of the
  $(n{\cdot}t_\phi)(E)$.  Let $x$ in $E$ be of infinite order. Then
  $y:=\alpha(x)$ is of infinite order as well. The point $t_\phi(x)$
  of the extension $\calP_x$ of $E$ by $\Gm$ has image $y$ in~$E$. The
  Zariski closure in $\calP_x$ of $\{n{\cdot}t_\phi(x) : n\in\Z\}$ is
  a closed subgroup $H$ of~$\calP_x$. The image of $H$ in $E$ is
  closed ($H\to E$ is a morphism of algebraic groups), and
  contains~$y$, hence is equal to~$E$. Hence $\dim(H)$ is $1$
  or~$2$. Assume that $\dim(H)=1$. By Lemma~\ref{lemma:fewsubgroups}
  the extension class of $\calP_x$ is torsion, but that contradicts
  that this class, being~$\lambda(x)$, is not torsion. We conclude
  that $\dim(H)=2$, and $H=\calP_x$. Hence $Z$ contains all $\calP_x$
  with $x$ not torsion. Then $Z=\calP$.
\end{proof}

\begin{Theorem}\label{thm:ttorsionproperty}
For every torsion point $x$ in $E$, $t_\phi(x)$ is torsion in~$\calP_x$.
\end{Theorem}
\begin{proof}
We will give three proofs: one in the context of abelian schemes and
biextensions (Proposition~\ref{prop:ribsectors}), one, more
elementary, using generalised jacobians of elliptic curves with a
double point in Section~\ref{sec6}, and a third proof, using the
description of $t_\phi(E)$ as a special subvariety of a mixed
Shimura variety (Proposition~\ref{prop_tors_prop_r_f_msv}).
We refer to~\cite[Section~1]{Bertrand11}, for the initial proof of
Theorem~\ref{thm:ttorsionproperty}, based on the theory of 1-motives. 
\end{proof}
We now explain why the closed subvariety
$Y:=t_\phi(E)$ in the family of semi-abelian varieties $B:=\calP$
over~$X:=E$ is a counter-example to~\cite[Conjecture~6.1]{Pink2} when
$\phi-\ol{\phi}\neq 0$. 
First of all, $Y$ is not contained in a proper subvariety of $B$ that
is a subgroup scheme of $B$ over $X$ because of
Lemma~\ref{lemma:tgeneric}. 

Secondly, $d:=\dim(Y)=1$, hence according to the conjecture, the
intersection of $Y$ with the set $B^{[>1]}$ that is the union, over
all $x$ in $X$, of all subgroups of $B_x$ of codimension $>1$, should
not be Zariski dense in~$Y$. However, $B^{[>1]}$ is the set of points
that are torsion in their fibre, and
Theorem~\ref{thm:ttorsionproperty} says that the intersection is
infinite. 

\section{The example with abelian schemes}\label{sec3}

In this section we consider abelian schemes, but even in the case of
elliptic curves, this section provides a new point of view on 
Ribet sections and their properties. We recommend Chapter~I of
\cite{Moret-Bailly} and references therein for further details about
biextensions, duality and pairings.

Let $S$ be a scheme, $A$ an abelian scheme over~$S$, and $A^\vee$ its
dual (Section~I.1 in~\cite{Faltings-Chai}). Let $\calL$ be the
universal line bundle on $A\times_S A^\vee$, rigidified, compatibly,
at $\{0\}\times A^\vee$ and~$A\times\{0\}$; it identifies $A$ with the
dual of~$A^\vee$. Then $\calP=\Isom_{A\times_S A^\vee}(\calO,\calL)$
is the Poincar\'e $\Gm$-torsor on $A\times_S A^\vee$, and as described
in the previous section in the case of elliptic curves, it is a
biextension of $A$ and $A^\vee$ by~$\Gm$.  In particular,
over~$A^\vee$, $\calP$ is the universal extension of $A$ by~$\Gm$, and
over~$A$, $\calP$ is the universal extension of $A^\vee$
by~$\Gm$. Proposition~\ref{prop:t_phi} extends to the present
situation as follows (see~\cite{Breen}, \cite{Chambert-Loir99},
\cite[Section~8.3]{Lopuhaa}). 

\begin{Proposition}\label{prop:ribsec}
  Let $S$ be a scheme, $A$ an abelian scheme over~$S$, $\calP$ the
  Poincar\'e torsor on $A\times_S A^\vee$, $f\colon A^\vee\to A$ a
  morphism of group schemes, $f^\vee\colon A^\vee\to (A^\vee)^\vee=A$
  its dual, and
\[
\alpha:=f-f^\vee\colon A^\vee\to A\,.
\]
The restriction of $\calP$ to the graph of $\alpha$ has a unique
section~$r_f$
\[
\begin{tikzcd}
{\Gm}_{A^\vee} \arrow[dr] \arrow[r, hook] 
& \calP \arrow[r, two heads] \arrow[d]
& A_{A^\vee} = A\times_SA^\vee \arrow[dl]\\
 & A^\vee \arrow[u, bend left, "r_f"] 
 \arrow[ur, bend right, "{(\alpha,\id)}"'] &
\end{tikzcd}
\]
with value $1$ at the origin.
\end{Proposition}
\begin{proof}
We start in a more general situation: let $A_1$ and $A_2$ be abelian
schemes over~$S$, $\calP_1$ and $\calP_2$ their Poincar\'e torsors,
and $f\colon A_1\to A_2$. Then the dual $f^\vee\colon A_2^\vee\to
A_1^\vee$ is defined by the condition that the pullback of the
universal extension
\[
\begin{tikzcd}
{\Gm}_{A_2^\vee} \arrow[r, hook] 
& \calP_2 \arrow[r, two heads] 
& (A_2)_{A_2^\vee} = A_2\times_SA_2^\vee 
\end{tikzcd}
\]
by $f\times \id\colon A_1\times_S A_2^\vee\to A_2\times_S A_2^\vee$ is
isomorphic to the pullback of the universal extension
\[
\begin{tikzcd}
{\Gm}_{A_1^\vee} \arrow[r, hook] 
& \calP_1 \arrow[r, two heads] 
& (A_1)_{A_1^\vee} = A_1\times_SA_1^\vee 
\end{tikzcd}
\]
by $\id\times f^\vee\colon A_1\times A_2^\vee\to A_1\times
A_1^\vee$. Such an isomorphism is unique, hence
\[
\text{for all $T\to S$, $x\in A_1(T)$, $y\in A_2^\vee(T)$:}
\quad \calP_1(x,f^\vee y) = \calP_2(fx,y)\,.
\]
Now we specialise to the case where $A_1=A_2^\vee$. Then
$A_1\times_S A_1^\vee=A_2^\vee\times_S A_2$, with Poincar\'e torsors
$\calP_1$ and $\sigma^*\calP_2$, where $\sigma\colon
A_2^\vee\times_SA_2\to A_2\times_S A_2^\vee$ is the coordinate switch. Then
we have, for $T\to S$, $x\in A_1(T)=A_2^\vee(T)$ and $y\in A_2^\vee(T)$:
\begin{equation}\label{eq:isombiexts}
\calP_2(fx,y) = \calP_1(x,f^\vee y) = \calP_2(f^\vee y,x)\,.
\end{equation}
Now we restrict to the case $y=x$, where we have
$\calP_2(fx,x)=\calP_2(f^\vee x,x)$. Then additivity in the first
factor gives that 
\begin{multline}\label{eq:ribsec}
\calP_2(\alpha x,x) = \calP_2((f-f^\vee)x,x) = \calP_2(fx - f^\vee x,x) \\
= \calP_2(fx,x)\otimes\calP_2(f^\vee x,x)^{-1} 
=\Hom(\calP_2(fx,x),\calP_2(f^\vee x,x)) = {\Gm}_T\, .
\end{multline}
Now we take $A_2=A$, and define $r_f\colon A^\vee \to \calP$ by
letting it send $x$ to the $T$-point of $\calP(\alpha x,x)$
corresponding to the unit section of ${\Gm}_T$ via the isomorphism
in~(\ref{eq:ribsec}). 

By construction, $r_f(0)=1$. This condition makes it unique, as two
such sections differ by a factor in $\calO(A^\vee)^\times=\calO(S)^\times$,
with value $1$ at~$0\in A^\vee(S)$. 
\end{proof}

\begin{Remark}
When $A \rightarrow S$ is a complex elliptic curve~$E$, and
$\lambda\colon E\to E^\vee$ is as in Section~\ref{sec2}, and $\phi$ is
in~$\End(E)$, and $f=\phi\circ\lambda$, then $t_\phi$ as
in~(\ref{eq:def_t_phi}) and $r_f$ as in Proposition~\ref{prop:ribsec} are
equal (well, up to switching the factors of $E\times E$), because they
are sections of the same $\Gm$-torsor over~$E$, with the same value
at~$0$. Therefore, Proposition~\ref{prop:ribsectors} below proves
Theorem~\ref{thm:ttorsionproperty}. 
\end{Remark}

The following Proposition gives the torsion property of $r_f$ at the
torsion points of~$A^\vee$: it implies that for $T\to S$ and $x$ in
$A^\vee[n](T)$ we have $n^2r_f(x)=1$. (See
Proposition~\ref{prop_tors_prop_r_f_msv}  and Theorem~\ref{thm:tJ} for
other proofs of this equality.) 

\begin{Proposition}\label{prop:ribsectors}
  Let $S$, $A$, $\calP$, $f$, $\alpha$ and $r_f$ be as in
  Proposition~\ref{prop:ribsec}. Let $n\geq1$, let $T$ be an
  $S$-scheme, and $x\in A^\vee[n](T)$. Then
\[
nr_f(x)=e_n(fx,x)\quad \text{in}\quad  \calP(n\alpha x,x)=\calP(0,x)=\Gm(T),
\]
with $e_n\colon A[n](T)\times A^\vee[n](T)\to \mu_n(T)$ the Weil
pairing (whose definition is recalled below).
\end{Proposition}
\begin{proof}
The base change $T\to S$ reduces to the case where $T=S$. First we
describe the Weil pairing in terms of~$\calP$. Let $z\in A[n](S)$ and
$y\in A^\vee[n](S)$. We have the following
canonical isomorphisms between $\Gm$-torsors on~$S$,
\[
\begin{tikzcd}
{\Gm}_S \arrow[r, equal] \arrow[d,"{e_n(z,y)}"] & \calP(z,0) \arrow[r, equal] 
& \calP(z,ny) \arrow[r,equal,"+_2"] 
& \calP(z,y) ^{\otimes n} \arrow[d,equal,"\id"] \\
{\Gm}_S \arrow[r, equal]& \calP(0,y) \arrow[r, equal] 
& \calP(nz,y) \arrow[r,equal,"+_1"] 
& \calP(z,y) ^{\otimes n}
\end{tikzcd}
\]
where the superscript $+_1$ means ``induced by additivity in the first
coordinate'', etc., and where $\calP(z,y)^{\otimes n}$ is the
contracted product of $n$ copies of~$\calP(z,y)$.  As the diagram
shows, we define $e_n(z,y)$ to be the image of the section $1$ of the
top ${\Gm}_S$ in the bottom~${\Gm}_S$. We claim that this is the usual
Weil pairing: let $\calP_y$ be the extension of $A$ by ${\Gm}_S$
at~$y$, then, as $n{\cdot}y=0$ in~$A^\vee(S)$, the pullback of the
extension
\[
\begin{tikzcd}
{\Gm}_S \arrow[r, hook] 
& \calP_y \arrow[r, two heads] 
& A
\end{tikzcd}
\]
by $n{\cdot}\colon A\to A$ splits (uniquely as for all extensions of
abelian schemes by affine group schemes), and so
there is a unique $\tilde{n}\colon A\to \calP_y$ that
lifts $n{\cdot}\colon A\to A$, and the restriction $\tilde{n}\colon
A[n]\to\mu_n$ sends $z$ to $e_n(z,y)$.

The following commutative diagram relates $nr_f(x)$ to $e_n(fx,x)$ and
$e_n(x,f^\vee x)$: going from bottom right to upper right and then
upper left is multiplication by $e_n(x,f^\vee x)$, going from bottom
right to middle right and then middle left and then upper left
is~$nr_f(x)$ by~(\ref{eq:ribsec}), and from bottom right to upper left
via bottom left is~$e_n(fx,x)$.
\begin{equation}\label{eq:ribsecweil}
\begin{tikzcd}
{\Gm}_S \arrow[r,equal, "\id"] \arrow[d,equal] \arrow[dr, phantom,
"\text{\scriptsize b}"]
& {\Gm}_S \arrow[r,equal,"\id"] \arrow[d,equal] \arrow[dr, phantom, "\text{\scriptsize  f}"]
& {\Gm}_S \arrow[d,equal] \\ 
\calP(0,x) \arrow[d,equal] \arrow[r,equal] \arrow[dr, phantom, "\text{\scriptsize  c}"
near start]
& (\sigma^*\calP)(0,f^\vee x) \arrow[d,equal] \arrow[r,equal] 
\arrow[dr, phantom, "\text{\scriptsize  g}"]
& \calP(f^\vee x,0) \arrow[d,equal] \\
\calP(fx,x)^{\otimes n} \arrow[d,equal] \arrow[r,equal] 
\arrow[dr, phantom, "\text{\scriptsize  d}"]
& (\sigma^*\calP)(x,f^\vee x)^{\otimes n} \arrow[d,equal]
\arrow[r,equal]  \arrow[dr, phantom, "\text{\scriptsize  h}"]
& \calP(f^\vee x,x)^{\otimes n} \arrow[d,equal]
\arrow[ll, bend right=13, crossing over, "{nr_f(x)}"' pos=0.65] \\
\calP(fx,0) \arrow[d,equal] \arrow[r,equal] 
\arrow[dr, phantom, "\text{\scriptsize  e}"]
& (\sigma^*\calP)(x,0) \arrow[d,equal] \arrow[r,equal]
\arrow[dr, phantom, "\text{\scriptsize  i}"]
& \calP(0,x) \arrow[d,equal] \\
{\Gm}_S \arrow[r,equal,"\id"'] \arrow[uuuu,bend left=60,
"{e_n(fx,x)}", very near end]
\arrow[uuuu, phantom, bend left=45, "\text{\scriptsize  a}"]
& {\Gm}_S \arrow[r,equal,"\id"'] 
& {\Gm}_S \arrow[uuuu,bend right=60,"{e_n(x,f^\vee x)}"', very near start]
\arrow[uuuu, phantom, bend right=45, "\text{\scriptsize  j}"]
\end{tikzcd}
\end{equation}
Here are arguments for the commutativity of all faces (a--j) in the diagram. 
\begin{description}
\item[a] This is the definition of $e_n(fx,x)$.
\item[b--e] This is because the equality signs in
  (\ref{eq:isombiexts}) are isomorphisms of biextensions on
  $A_2^\vee\times_S A_2^\vee$.
\item[f--i] These follow directly from the definition
  of~$\sigma^*\calP$.
\item[j] This is the definition of $e_n(x,f^\vee x)$.
\end{description}
Let us remark that the commutativity of this diagram shows that
$f^\vee$ and $f$ are adjoints for the $e_n$-pairing, and that when
$f^\vee=f$, $e_n(fx,x)=1$ for all $x$ in $A^\vee[n](S)$, in
particular, that the pairings attached to a polarisation are alternating.
\end{proof}

\section{The Poincar\'e torsor as mixed Shimura variety}\label{sec4}
In this section we describe the Poincar\'e torsor of the universal
family of principally polarised complex abelian varieties of
dimension~$d$ as a mixed Shimura variety, that is, as a moduli space
for mixed Hodge structures. We recommend \cite[Section~2]{Pink1} (and
also~\cite{Klingler} and~\cite{HodgeII}) as an introduction to mixed
Hodge structures and (connected) mixed Shimura varieties, but we do
not assume the reader to be familiar with these notions. In fact, we
hope that the example treated here also provides a good introduction,
and perhaps a motivation to read more. We find that the point of view
of mixed Shimura varieties gives a simple and beautiful perspective on
the uniformisation of the universal Poincar\'e torsor. The notion of
1-motives from~\cite{HodgeIII} provides an algebraic description of
the mixed Hodge structures that we encounter, but we will not use this.

\subsection{Pure Hodge structures}
\label{sec:phs}
For $n$ in~$\Z$, a \emph{$\Z$-Hodge structure of weight~$n$} is a
finitely generated $\Z$-module $M$ together with a decomposition
(called Hodge decomposition) of the complex vector
space~$M_\C:=\C\otimes M$:
\[
  M_\C = \bigoplus_{\substack{p,q\in\Z\\p+q=n}} M^{p,q},
\]
such that for all $p,q$ in $\Z$ with $p+q=n$:
\[
  \quad M^{q,p}=\ol{M^{p,q}}, 
\]
where $\ol{M^{p,q}}$ is the image of $M^{p,q}$ under the map
$M_\C\to M_\C$ that sends $z\otimes m$ to $\ol{z}\otimes m$.  A
\emph{pure $\Z$-Hodge structure} (also called \emph{split mixed
  $\Z$-Hodge structure}) is a finitely generated $\Z$-module~$M$,
together with a direct sum decomposition
\[
  M/M_\tors=\bigoplus_{n\in\Z}M_n,
\]
and for each $n$ a Hodge structure of weight~$n$,
\[
  M_{n,\C}=\bigoplus_{p+q=n}M^{p,q}.
\]
For $T\subset\Z^2$, $M$ is said to be of type~$T$, if, for all $(p,q)$
not in~$T$, $M^{p,q}$ is zero. 

A morphism of pure $\Z$-Hodge structures
\[
\begin{tikzcd}
  (M,(M^{p,q})_{p,q}) \ar[r] & (N,(N^{p,q})_{p,q})    
\end{tikzcd}
\]
is a morphism $f\colon M\to N$ of $\Z$-modules
such that for all $(p,q)$ one has $f_\C(M^{p,q})\subset N^{p,q}$.

For $M$ and $N$ pure $\Z$-Hodge structures, $M^\vee$, $M\otimes N$ are
given pure $\Z$-Hodge structures as follows:
\[
(M^\vee)^{p,q} = (M^{-p,-q})^\vee\,,\quad 
(M\otimes N)^{p,q} = \bigoplus_{\substack{a+c=p\\ b+d=q}}
(M^{a,b}\otimes N^{c,d})\,,
\]
and this dictates the rule for~$\Hom(M,N)$:
\[
\Hom(M,N)^{p,q} = (M^\vee\otimes N)^{p,q}
=\bigoplus_{\substack{-a+c=p\\ -b+d=q}}
\Hom(M^{a,b},N^{c,d})\,.
\]

It is convenient to define, for $m$ in~$\Z$, the $\Z$-Hodge structure
$\Z(m)$ of weight $-2m$ as the sub-$\Z$-module $(2\pi i)^m\Z$ of~$\C$,
with $\Z(m)_\C=\Z(m)^{-m,-m}$. For $M$ a pure $\Z$-Hodge structure,
and $m$ in $\Z$, $M(m)$ denotes $M\otimes\Z(m)$. The embedding $(2\pi
i)^m\Z\subset\C$ gives the isomorphisms $\Z(m)_\C=\C$ and
$M(m)_\C=M_\C$.

A \emph{polarisation} on a pure $\Z$-Hodge structure $M$ of weight $n$
is a morphism of pure $\Z$-Hodge structures $\Psi\colon M\otimes
M\to\Z(-n)$ such that for every $(p,q)$ with $p+q=n$ the map
\[
M^{p,q}\times M^{p,q}\to \C,\quad (v,w)\mapsto (-1)^p\Psi(v,\ol{w})
\]
is a complex inner product (that is, for all $(v,w)$,
$\Psi(w,\ol{v})=\ol{\Psi(v,\ol{w})}$, and, for all $v\neq0$,
$(-1)^p\Psi(v,\ol{v})>0$).  The symmetry condition is equivalent to 
$\Psi$ being symmetric if $n$ is even and
antisymmetric if $n$ is odd. The symmetry and positivity conditions are
equivalent to the restriction to $M_\R\times M_\R$ of the
$\C$-bilinear map  
\[
M_\C\times M_\C\to\C\,,\quad (x,y)\mapsto (2\pi i)^n\Psi(x\otimes i{\cdot}y)
\]
with $i$ acting on $M^{p,q}$ as multiplication by $i^{-p}\ol{i}^{-q}$
being  $\R$-valued, symmetric and positive definite.

\subsection{Principally polarised abelian varieties}
\label{sec:ppavs}

Let $d$ be in $\Z_{\geq1}$. Principally polarised complex abelian
varieties of dimension~$d$ are conveniently described as
follows. Their lattice is a free $\Z$-module $M$ of rank $2d$ with a
Hodge structure $M_\C=M^{-1,0}\oplus M^{0,-1}$, and the polarisation
$\Psi\colon M\otimes M\to \Z(1)=2\pi i\Z$ is antisymmetric and induces
an isomorphism $M\to M^\vee(1)$. The abelian variety is then
$M_\C/(M^{0,-1}+M)$. Then $M$ together with $\Psi$ is isomorphic to
$\Z^{2d}$ with $\Psi\colon \Z^{2d}\otimes\Z^{2d}\to\Z(1)$, $x\otimes
y\mapsto 2\pi ix^t(\begin{smallmatrix} 0 & -1\\ 1 &
  0\end{smallmatrix})y$, and such an isomorphism is unique up to
composition with an element of~$\Sp(\Psi)(\Z)$ (the stabiliser of
$\Psi$ in~$\GL_{2d}(\Z)$). Let $(e_1,\ldots,e_{2d})$ be the standard
basis of~$\Z^{2d}$. The subspace $M^{0,-1}$ of~$\C^{2d}$, on which
$(v,w)\mapsto \Psi(v,\ol{w})$ is an inner product, has trivial
intersection with the isotropic subspaces generated by
$e_1,\ldots,e_d$ and $e_{d+1},\ldots,e_{2d}$, hence there is a unique
$\tau$ in $\GL_d(\C)$ such that $M^{0,-1}=\{(\substack{\tau v\\ v}) :
v\in\C^d\}$. As $\Psi$ is a morphism of Hodge structures, $M^{0,-1}$
is isotropic for $\Psi$, giving $\tau^t=\tau$. The positivity of the
complex inner product on $M^{0,-1}$ gives that
$\Im(\tau)=(\tau-\ol{\tau})/2i$ is positive definite. Conversely, for
every $\tau\in\rmM_d(\C)$ with $\tau^t=\tau$ and $\Im(\tau)$ positive
definite, $\tau$ is in $\GL_d(\C)$ and $M^{0,-1}:=\{(\substack{\tau v\\
  v}) : v\in\C^d\}$ gives a Hodge structure on $\Z^{2d}$ such that
$\Psi$ is a principal polarisation.

We conclude: the set $D_\Psi$ of Hodge structures of type
$\{(-1,0),(0,-1)\}$ on $\Z^{2d}$ for which $\Psi$ is a polarisation is
in bijection with the Siegel half space $\H_d$ of symmetric
$\tau\in\rmM_d(\C)$ with $\Im(\tau)$ positive definite, via
$\tau\mapsto M^{0,-1}_\tau:=\{(\substack{\tau v\\ v}) :
v\in\C^d\}$. Note that $\H_d$ is a convex open subset of the set of
symmetric $d$ by $d$ complex matrices. The action of $\Sp(\Psi)(\Z)$
describes the moduli of complex principally polarised abelian
varieties of dimension~$d$: the quotients by suitable congruence
subgroups give fine moduli spaces, and the stacky quotient by
$\Sp(\Psi)(\Z)$ gives the stack of complex principally polarised
abelian varieties of dimension~$d$. Let us write more explicitly the
abelian variety $A_\tau:=\C^{2d}/(M^{0,-1}_\tau+\Z^{2d})$ at $\tau$
in~$\H_d$. The $\C$-linear map $\C^{2d}\to\C^d$, $(\substack{w\\
  v})\mapsto w-\tau v$ is surjective and has kernel~$M^{0,-1}$. So
$A_\tau$ is the cokernel of $(\begin{matrix} 1_d &
  -\tau\end{matrix}){\cdot} \colon\Z^{2d}\to \C^d$, $(\substack{x\\
  y})\mapsto x-\tau y$, that is, $A_\tau$ is the quotient of $\C^d$ by
the lattice generated by $\Z^d$ and the columns of~$\tau$.

For all $M^{0,-1}$ in $D_\Psi$ and all $g$ in $\GL_{2d}(\R)$,
$gM^{0,-1}$ is a Hodge structure of type $\{(-1,0),(0,-1)\}$ for which
$g\Psi$ is a polarisation, where, for all $x,y$ in $\R^{2d}$,
$(g\Psi)(x\otimes y) = \Psi((g^{-1}x)\otimes(g^{-1}y))$.  Hence
$\Sp(\Psi)(\R)$, the subgroup of $\GL_{2d}(\R)$ that preserves~$\Psi$,
acts on~$D_\Psi$.

The following argument shows that this action is transitive. Let
$M^{0,-1}$ be in~$D_\Psi$, and let $v_1,\ldots,v_d$ be an orthonormal
basis for~$M^{0,-1}$. Then the $2d$ elements of~$\R^{2d}$,
$\Re(v_1),\ldots,\Re(v_d),\Im(v_1),\ldots,\Im(v_d)$, form an $\R$-basis
of $\R^{2d}$ with respect to which $M^{0,-1}$ and $\Psi$ do not depend
on~$M^{0,-1}$: indeed, $M^{0,-1}\subset\C^{2d}$ is the $\C$-subspace
generated by
$\Re(v_1)+i\Im(v_1),\ldots, \Re(v_d)+i\Im(v_d)$, and, for every~$j$, we have
that $\Psi(\Re(v_j),\Im(v_j))=i/2$ and all other $\Psi(\Re(v_j),\Im(v_k))$
and $\Psi(\Re(v_j),\Re(v_k))$ and $\Psi(\Im(v_j),\Im(v_k))$ are zero.

In fact a slightly bigger group acts on~$D_\Psi$.  We view $\Psi$ as
an element of the $\R$-vector space
$(\R^{2d}\otimes_\R\R^{2d})^\vee\otimes_\R\R(1)$, on which the group
$\GL_{2d}(\R)\times\R^\times$ acts. An element $(g,\lambda)$ acts as
$(g^{-1}\otimes g^{-1})^\vee\otimes \lambda$. Then $(g,\lambda)$ fixes
$\Psi$ if and only if for all $x,y\in\R^{2d}$,
$\Psi(gx,gy)=\lambda\Psi(x,y)$. We let $\GSp_\Psi(\R)$ be the group of
such $(g,\lambda)$, and $\GSp_\Psi(\R)^+$ the subgroup of the
$(g,\lambda)$ with $\lambda>0$. Then $\GSp_\Psi(\R)^+$ acts on $D_\Psi$
via $M^{0,-1}\mapsto g{\cdot}M^{0,-1}$.

\subsection{Mixed Hodge structures}
\label{sec:mhs}
A \emph{mixed Hodge structure} on a finitely generated $\Z$-module $M$
is the data of an increasing filtration $(W_nM)_{n\in\Z}$ (called the
weight filtration) with $W_nM=M_\tors$ for $n$ small enough and
$W_nM=M$ for $n$ large enough, with all $M/W_nM$ torsion free, and a
decreasing filtration $(F^pM_\C)_{p\in\Z}$ of the $\C$-vector space
$M_\C$, with $F^pM_\C=M_\C$ for small enough $p$ and $F^pM_\C=0$ for
large enough~$p$, such that for each $n$ in $\Z$ the filtration
induced by $F$ on $(\Gr^W_nM)_\C:=((W_nM)/(W_{n-1}M))_\C$ is a Hodge
structure of weight~$n$:
\[
  (\Gr^W_nM)_\C = \bigoplus_{p+q=n} (\Gr^W_nM)^{p,q}_\C,
\]
with
\[
  (\Gr^W_nM)^{p,q}_\C = F^p(\Gr^W_nM)_\C\cap \ol{F^q (\Gr^W_nM)_\C}.
\]
Let us determine all mixed Hodge structures on
$M:=\Z{\cdot}e_1\oplus\Z{\cdot}e_2$, with $W_{-3}(M)=0$,
$W_{-2}(M)=W_{-1}(M)=\Z{\cdot}e_1$ and $W_0(M)=M$, of type
$\{(-1,-1),(0,0)\}$, that is, extensions of $\Z(0)$ by~$\Z(1)$. Then
$F^{-1}M_\C=M_\C$, $F^1M_\C=0$, and $F^0M_\C\cap\C{\cdot}e_1=0$ and
under the quotient map $q\colon M_\C\to M_\C/W_{-1}M_\C=\C{\cdot}e_2$,
$F^0M_\C$ is mapped surjectively. So $F^0M_\C$ is a line, of the form
$L_a:=\C{\cdot}(e_2+ae_1)$ for a unique $a$ in~$\C$, giving a
bijection from $\C$ to the set $D_W$ of mixed Hodge structures of the
type we consider.

Let $P_W(\R)$ be the subgroup of
$\GL_2(\R)\times\GL(\R(1))\times\GL(\R(0))$ that fixes $\R(1)\to\R^2$,
$2\pi i\mapsto e_1$, that fixes $\R^2\to\R(0)$, $(x,y)\mapsto y$, and
that fixes $\R(0)\otimes\R(0)\to\R(0)$, $x\otimes y\mapsto xy$. Then
\[
P_W(\R) = 
\left\{\left(\begin{pmatrix}\lambda & x\\0&1\end{pmatrix}, \lambda, 1\right)
: \text{$\lambda\in\R^\times$, $x\in\R$}\right\}.
\]
By definition $P_W(\R)$ acts on $D_W$, and transported to $\C$ this
action is given by $a\mapsto \lambda a+x$. This action has two orbits:
$\R$ and~$\C-\R$. We would like to have a transitive action (in order
to get a ``connected mixed Shimura datum'' as in~\cite[Def.~2.1]{Pink1}). To get
that, we allow $x$ to be complex, that is, we let $U_W(\C)$ be the subgroup
of $\GL_2(\C)$ of unipotent matrices 
$(\begin{smallmatrix}1 & x\\0&1\end{smallmatrix})$ with $x\in\C$, and
let 
\[
P_W(\R)U_W(\C) = 
\left\{\left(\begin{pmatrix}\lambda & x\\0&1\end{pmatrix}, \lambda, 1\right)
: \text{$\lambda\in\R^\times$, $x\in\C$}\right\}
\]
act on~$D_W$. The action of $P_W(\Z)$ on $\C$ describes the moduli of
mixed $\Z$-Hodge structures that are extension of $\Z(0)$
by~$\Z(1)$. The coarse moduli space is the quotient
\[
  \C\to\C^\times\to \C\,, \quad
  a\mapsto \exp(2\pi ia)\mapsto \exp(2\pi ia)+\exp(-2\pi ia)\,.
\]

\subsection{The universal Poincar\'e torsor as moduli space of mixed
  Hodge structures}
\label{sec:uptams}
Let $d$ be in $\Z_{\geq1}$ and
\[
  M:=\Z(1)\oplus\Z^{2d}\oplus\Z\,,
\]
with
standard basis $2\pi i e_0,e_1,\ldots,e_{2d+1}$, and with the
following filtration: 
\[
\begin{aligned}
W_{-3}M & =\{0\},\quad W_{-2}M=\Z{\cdot}2\pi ie_0, \\
W_{-1}M & =\Z{\cdot}2\pi ie_0\oplus\cdots\oplus\Z{\cdot}e_{2d}, 
\quad W_{0}M=M.    
\end{aligned}
\]
Let $D$ be the set of filtrations $F$ on $M_\C$ such that $(M,W,F)$ is
a mixed $\Z$-Hodge structure of type $\{(-1,-1), (-1,0),(0,-1),
(0,0)\}$, and such that $\Psi\colon (x,y)\mapsto 2\pi i
x^t(\begin{smallmatrix}0 & -1\\1&0\end{smallmatrix})y$ is, via the
given bases, a polarisation on~$\Gr^W_{-1}M$. For $F$ in $D$ we have
$F^{-1}M_\C=M_\C$, and $F^1M_\C=\{0\}$, so $F$ is given
by~$F^0M_\C$. We get a map from $D$ to the set $D_\Psi$
(see Section~\ref{sec:ppavs}) by sending $F^0$ to
$F^0(\Gr^W_{-1}M_\C)$. Recall that we have a bijection $\H_d\to
D_\Psi$ that sends $\tau$ to $M^{0,-1}_\tau=(\substack{\tau \\
  1_d})\C^d\subset\C^{2d}$.

For $m$ and $n$ in $\Z_{\geq0}$ we denote by $\rmM_{m,n}(\C)$ the set
of complex $m$ by $n$ matrices.

\begin{Proposition}\label{prop:paramD}
There is a bijection
$\H_d\times \rmM_{1,d}(\C) \times \rmM_{d,1}(\C)\times \C \lto D$,  
\[
(\tau,u,v,w)\mapsto 
\begin{pmatrix}
u & w\\ \tau & v \\ 1_d & 0\\ 0 & 1
\end{pmatrix}
\C^{d+1} \subset M_\C=\bigoplus_{j=0}^{2d+1}\C e_j\,.
\]
\end{Proposition}
\begin{proof}
Let $\tau$ be in~$\H_d$. The $F^0(W_{-1}(M)_\C)$ in the fibre over
$\tau$ are the subspaces of $W_{-1}(M)_\C$
that are mapped isomorphically to the subspace $M^{0,-1}_\tau$
of~$\Gr^W_{-1}(M)_\C$ in the short exact sequence
\[
0 \to  W_{-2}(M)_\C \to W_{-1}(M)_\C \to \Gr^W_{-1}(M)_\C \to 0.
\]
This accounts for the first $d$ columns in the matrix above. We take these
columns as the first $d$ elements of our basis of~$F^0M_\C$.

Each $F^0(M_\C)$ in $D$ that restricts to $F^0(W_{-1}M_\C)$ given by
a $(\tau,u)$ has a unique $d{+}1$th basis vector
$\sum a_i e_i$ ending with $d$ zeros and then a~$1$. This accounts for
the last column.
\end{proof}

Let $P$ be the subgroupscheme of $\GL(M)\times\GL(\Z(1))$ that fixes
$W$, $\Z(1)\to W_{-2}(M)$, $2\pi i a\mapsto 2\pi i ae_0$,
$\Z(0)\to\Gr^W_0(M)$, $a\mapsto ae_{2d+1}$, and $\Psi\colon
\Gr^W_{-1}(M)\otimes \Gr^W_{-1}(M)\to\Z(1)$. Then, for any
$\Z$-algebra $R$ (we will only use $\Z$, $\R$ and~$\C$), we have
\begin{equation}
\label{eq:def_P}
P(R) = 
\left\{
\begin{pmatrix}
\mu(g) & x & z\\
0 & g & y\\
0 & 0 & 1    
\end{pmatrix} :
\begin{array}{c}
\text{$(g,\mu(g))\in\GSp(\Psi)(R)$,}\\ \\
\text{$x\in\rmM_{1,2d}(R)$, $y\in\rmM_{2d,1}(R)$, $z\in R$}
\end{array}
\right\},
\end{equation}
where the matrices are with respect to the $\Z$-basis $2\pi ie_0,
e_1,\ldots,e_{2d+1}$ of~$M$. We let $U$ be the subgroupscheme of $P$
given by 
\[
U(R) = 
\left\{
\begin{pmatrix}
1 & 0 & z\\
0 & 1 & 0\\
0 & 0 & 1    
\end{pmatrix} :
\text{$z\in R$}
\right\}.
\]
We also let $P^u$ be the
unipotent radical of~$P$, that is, 
\[
P^u(R) = 
\left\{
\begin{pmatrix}
1 & x & z\\
0 & 1 & y\\
0 & 0 & 1    
\end{pmatrix} :
\text{$x\in\rmM_{1,2d}(R)$, $y\in\rmM_{2d,1}(R)$, $z\in R$}
\right\},
\]
also known as the Heisenberg group.
Then $P^u$ is a central extension of the vector group $P^u/U$
by~$\Ga$. The commutator pairing on $P^u/U$ sends
$((x,y),(x',y'))$ to $xy'-x'y$.

For $R$ a subring of~$\C$, the matrix with respect to the $\C$-basis
$e_0,\ldots,e_{2d+1}$ of $M_\C$ of the element of $P(R)$ above is
\begin{equation}
\label{eq:P_element}
\begin{pmatrix}
\mu(g) & 2\pi i x & 2\pi iz\\
0 & g & y\\
0 & 0 & 1    
\end{pmatrix}\,.  
\end{equation}
By definition, $P(\R)^+U(\C)$ acts on~$D$. We make this explicit for
elements of $P^u(\R)U(\C)$, with respect to the $\C$-basis
$e_0,\ldots,e_{2d+1}$, writing $2\pi ix= (2\pi ix_1\,\, 2\pi ix_2)$ and
$y=(\substack{y_1\\ y_2})$: 
\begin{multline}\label{eq:P-action}
\begin{pmatrix}
1 & 2\pi ix_1 & 2\pi ix_2 & 2\pi iz\\
0 & 1_d & 0 & y_1\\
0 & 0 & 1_d & y_2\\
0 & 0 & 0 & 1    
\end{pmatrix}
\begin{pmatrix}
u & w\\ \tau & v \\ 1_d & 0\\ 0 & 1
\end{pmatrix}
\C^{d+1} \\ =
\begin{pmatrix}
u+2\pi ix_1\tau + 2\pi ix_2& w+2\pi ix_1v+2\pi iz\\ 
\tau & v + y_1\\ 1_d & y_2\\ 0 & 1
\end{pmatrix}
\C^{d+1} \\ =
\begin{pmatrix}
u+2\pi ix_1\tau + 2\pi ix_2& w+2\pi ix_1v+2\pi iz\\ 
\tau & v + y_1\\ 1_d & y_2\\ 0 & 1
\end{pmatrix}
\begin{pmatrix}
1_d & -y_2\\ 0 & 1
\end{pmatrix}
\C^{d+1} \\ =  
\begin{pmatrix}
u+2\pi ix_1\tau + 2\pi ix_2& w+2\pi ix_1v+2\pi iz-(u+2\pi ix_1\tau + 2\pi ix_2)y_2\\ 
\tau & v + y_1 -\tau y_2\\ 
1_d & 0
\\ 0 & 1
\end{pmatrix}
\C^{d+1}\,.
\end{multline}

As the action of $\Sp_\Psi(\R)$ on $D_\Psi$ is transitive, we conclude
that the action of $P(\R)^+U(\C)$ on $D$ is transitive. We also write
out the action of $\GSp_\Psi(\R)^+$ on~$D$:
\begin{multline}\label{eq:GSp-action}
\begin{pmatrix}
\mu & 0 & 0 & 0\\
0 & a & b & 0\\
0 & c & d & 0\\
0 & 0 & 0 & 1    
\end{pmatrix}
\begin{pmatrix}
u & w\\ \tau & v \\ 1_d & 0\\ 0 & 1
\end{pmatrix}
\C^{d+1} =
\begin{pmatrix}
\mu u & \mu w\\ a\tau +b & av \\ c\tau+ d & cv\\ 0 & 1
\end{pmatrix}
\begin{pmatrix}
(c\tau+d)^{-1} & 0 \\ 0 & 1
\end{pmatrix}
\C^{d+1}\\
 = 
\begin{pmatrix}
\mu u (c\tau+d)^{-1} & \mu w\\ (a\tau +b) (c\tau+d)^{-1} & av \\ 1_d & cv\\ 0 & 1
\end{pmatrix}
\begin{pmatrix}
1_d & -cv \\0 & 1
\end{pmatrix}
\C^{d+1}\\ = 
\begin{pmatrix}
\mu u (c\tau+d)^{-1} & \mu w-\mu u (c\tau+d)^{-1}cv\\ 
(a\tau +b) (c\tau+d)^{-1} & av -(a\tau +b) (c\tau+d)^{-1} cv\\ 1_d & 0\\ 0 & 1
\end{pmatrix} \C^{d+1}\,.
\end{multline}

\begin{Proposition}\label{prop:MSV_and_PT}
The quotient $P^u(\Z)\backslash D$ is the universal Poincar\'e torsor over~$\H_d$.
\end{Proposition}
\begin{proof}
  We prove this by showing that the universal extension of the
  universal abelian variety over $\H_d$ by $\C^\times$ is uniformised
  in exactly the same way when we express everything in terms of
  matrices. We view $\rmM_{1,d}(\C)$ and $\rmM_{d,1}(\C)$ as duals via
  the matrix multiplication (row times column).

Let us first consider a complex torus $A=V/L$, and an extension of
complex Lie groups
\[
0 \to \C^\times \to E \to A \to 0\,.
\]
Passing to universal covers gives us an extension of $\C$-vector
spaces
\[
0 \to \C \to \tilde{E} \to V \to 0\,,
\]
mapping to the previous sequence by exponential maps. The kernels of
these maps form an extension
\[
0 \to \Z(1) \to M \to L \to 0\,.
\]
The extensions of $V$ by $\C$ and of $L$ by $\Z(1)$ admit splittings,
and these are unique up to $V^\vee:=\Hom_\C(V,\C)$ and
$\Hom_\Z(L,\Z(1))=L^\vee(1)$. It follows that all extensions of $A$ by
$\C^\times$ are obtained as cokernels of maps
\begin{equation}\label{eq:cokernels}
\begin{aligned}
  \Z(1)\oplus L & \to \C\oplus V\,,\\
(2\pi in,m)& \mapsto (2\pi i n - \alpha(m),m),\quad \text{with
  $\alpha\in\Hom_\Z(L,\C)=L_\C^\vee$.}     
\end{aligned}
\end{equation}
Our reason for choosing $2\pi i n-\alpha(m)$ in the line above, and
not $2\pi i n+\alpha(m)$, is to avoid a sign in the isomorphism under
construction between our universal extension here and that given by
$P^u(\Z)\backslash D$; see the term $-uy_2$ in the upper right
coefficient in the last matrix in~(\ref{eq:P-action}).

More explicitly, over $L_\C^\vee$ we have a family of extensions, with
fibre at $\alpha$ the cokernel above. This family is universal for
extensions with given splitting of their tangent spaces at~$0$ and
given splitting of the kernel of the exponential map. On it, we have
actions of $V^\vee$ and $L^\vee(1)$, the quotient by which gives us
the universal extension of $A$ by $\C^\times$, with base
$L_\C^\vee/(V^\vee+L^\vee(1))$, which is therefore the dual complex
torus. The family itself is the quotient of $L_\C^\vee\times
V\times\C$ by a joint action of $V^\vee$, $L^\vee(1)$, $L$
and~$\Z(1)$.  By ``joint action'' we mean that the actions of the
individual elements of these four groups taken in this order induce a
group structure on $V^\vee\times L^\vee(1)\times L\times\Z(1)$ and an
action by that group on $L_\C^\vee\times V\times\C$. We make this more
explicit for the family over~$\H_d$.

Let $\tau$ be in~$\H_d$. As in Section~\ref{sec:ppavs} we have 
\[
\begin{aligned}
  A_\tau& =\C^{2d}/((\substack{\tau\\ 1_d})\C^d+\Z^{2d}) =
  \C^d/((1_d\,\,-\tau)\Z^{2d}) \\
  & = \rmM_{d,1}(\C)/((1_d\,\,-\tau)\rmM_{2d,1}(\Z)).    
\end{aligned}
\]
The universal extension of $A_\tau$ by $\C^\times$ is the quotient of
$\rmM_{1,2d}(\C)\times \rmM_{d,1}(\C)\times\C$ by the
joint actions of the groups 
$\rmM_{1,d}(\C)$, $\rmM_{2d,1}(\Z)$, $\rmM_{1,2d}(\Z(1))$,
and~$\Z(1)$.
We admit that this is not the same order as a few lines above, but the
rest of the proof shows that once the quotient by $\rmM_{1,d}(\C)$ has
been taken, the remaining three groups match the corresponding pieces of
the Heisenberg group, and therefore the order in which we consider their
actions is irrelevant.

An element $l$ in $\rmM_{1,d}(\C)$ acts by postcomposing the embedding
of $\Z(1)\oplus\rmM_{2d,1}(\Z)$ in $\C\oplus \rmM_{d,1}(\C)$ as
in~(\ref{eq:cokernels}) with 
\[
\begin{pmatrix}
w \\v
\end{pmatrix}
\mapsto
\begin{pmatrix}
1 & l\\0 &1_d
\end{pmatrix}
\begin{pmatrix}
w \\v
\end{pmatrix}
=
\begin{pmatrix}
w +lv\\v
\end{pmatrix}
\]
giving the embedding
\[
\begin{aligned}
\begin{pmatrix}
2\pi in\\ m_1\\ m_2
\end{pmatrix}
& \mapsto
\begin{pmatrix}
1 & l\\0 & 1_d
\end{pmatrix}
\begin{pmatrix}
1 & -\alpha_1 & -\alpha_2\\0 & 1_d & -\tau
\end{pmatrix}{\cdot}
\begin{pmatrix}
2\pi in\\ m_1\\ m_2
\end{pmatrix} \\
& =
\begin{pmatrix}
1 & -\alpha_1+l & -\alpha_2-l\tau\\0 & 1_d & -\tau
\end{pmatrix}{\cdot}
\begin{pmatrix}
2\pi in\\ m_1\\ m_2
\end{pmatrix}\,.    
\end{aligned}
\]
The two displayed formulas above give the actions of $l$ on $(v,w)$ in
$\rmM_{d,1}(\C)\times\C$ and on $(\alpha_1,\alpha_2)$ in
$\rmM_{1,2d}(\C)$, and therefore the action on
$\rmM_{1,2d}(\C)\times\rmM_{d,1}(\C)\times\C$
\[
l\colon (\alpha_1,\alpha_2,v,w)\mapsto (\alpha_1-l,\alpha_2+l\tau,v,w+l(v)).
\]
We make a quotient map for this action as follows. For every
$(\alpha_1,\alpha_2,v,w)$ there is a unique $l$, namely, $\alpha_1$,
that brings it to the subset of all $(0,\alpha_2,v,w)$.  This gives us
the quotient map
\[
\begin{aligned}
q\colon & \rmM_{1,2d}(\C)\times\rmM_{d,1}(\C)\times\C \to 
\rmM_{1,d}(\C)\times\rmM_{d,1}(\C)\times\C,\\
& (\alpha_1,\alpha_2,v,w)\mapsto (\alpha_1\tau+\alpha_2,v,w+\alpha_1v)\,,
\end{aligned}
\]
whose target is the source at $\tau$ of the bijection in Proposition~\ref{prop:paramD}.
Now we consider the other actions and push them to this quotient.

At the point $(\alpha_1\,\,\alpha_2)$ in
$\rmM_{1,2d}(\C)$ the embedding of $\Z(1)\oplus\rmM_{2d,1}(\Z)$ in
$\C\oplus \rmM_{d,1}(\C)$ is
\[
\begin{pmatrix}
2\pi in\\ m_1\\ m_2
\end{pmatrix}
\mapsto
\begin{pmatrix}
1 & -\alpha_1 & -\alpha_2\\0 & 1_d & -\tau
\end{pmatrix}{\cdot}
\begin{pmatrix}
2\pi in\\ m_1\\ m_2
\end{pmatrix}\,,
\]
and therefore $(2\pi in, (\substack{m_1\\m_2}))$ in
$\Z(1)\times\rmM_{2d,1}(\Z)$ acts on $\rmM_{1,2d}(\C)\times
\rmM_{d,1}(\C)\times\C$ by the translations
\[
\begin{tikzcd}
(\alpha_1,\alpha_2,v,w) \arrow[rr, mapsto,"{(2\pi in, (\substack{m_1\\m_2}))}"]&&
(\alpha_1,\alpha_2,v+m_1-\tau m_2,w+2\pi in-\alpha_1m_1-\alpha_2m_2)\,.
\end{tikzcd}
\]
It follows that $2\pi i n$ and $(\substack{m_1\\m_2})$ act on
$\rmM_{1,d}(\C)\times\rmM_{d,1}(\C)\times\C$ by
\begin{equation}
\label{eq:action_m}
\begin{aligned}
2\pi i n\colon (u,v,w) & \mapsto (u,v,w+2\pi i n)\,,\\
(\substack{m_1\\m_2})\colon (u,v,w) & \mapsto (u,v+m_1-\tau m_2,w-um_2)\,.  
\end{aligned}
\end{equation}
An element $2\pi i(n_1\,\, n_2)$ in $\rmM_{1,2d}(\Z(1))$ acts by precomposing
the embedding
\[
\begin{tikzcd}
\Z(1)\oplus\rmM_{2d,1}(\Z) \ar[r,hook] & \C\oplus \rmM_{d,1}(\C)
\end{tikzcd}
\]
with
\[
\begin{aligned}
\begin{pmatrix}
2\pi in \\ m_1\\ m_2
\end{pmatrix}
& \mapsto
\begin{pmatrix}
1 & -2\pi in_1 & -2\pi i n_2\\
0 & 1_d & 0 \\
0 & 0 & 1_d
\end{pmatrix} {\cdot}
\begin{pmatrix}
2\pi in \\ m_1\\ m_2
\end{pmatrix} \\
& =
\begin{pmatrix}
2\pi i(n - n_1m_1 - n_2m_2) \\ m_1\\ m_2
\end{pmatrix}\,.
\end{aligned}
\]
where we have introduced a factor $-1$ because we want a left
action. This gives the embedding
\[
\begin{aligned}
\begin{pmatrix}
2\pi in \\ m_1\\ m_2
\end{pmatrix}
& \mapsto
\begin{pmatrix}
1 & -\alpha_1 & -\alpha_2\\0 & 1_d & -\tau
\end{pmatrix}{\cdot}
\begin{pmatrix}
1 & -2\pi in_1 & -2\pi i n_2\\
0 & 1_d & 0 \\
0 & 0 & 1_d
\end{pmatrix} {\cdot}
\begin{pmatrix}
2\pi in \\ m_1\\ m_2
\end{pmatrix} \\
& = 
\begin{pmatrix}
1 & -\alpha_1-2\pi i n_1 & -\alpha_2-2\pi i n_2\\0 & 1_d & -\tau
\end{pmatrix}{\cdot}
\begin{pmatrix}
2\pi in \\ m_1\\ m_2
\end{pmatrix}\,.
\end{aligned}
\]
So the identity on $\C\oplus\rmM_{d,1}(\C)$ and the inverse of the
action of $2\pi i(n_1\,\, n_2)$ on $\Z(1)\oplus\rmM_{2d,1}(\Z)$ induce
an isomorphism from the extension at $(\alpha_1,\alpha_2)$ to the
extension at $(\alpha_1+2\pi i n_1, \alpha_2+2\pi i n_2)$. 
Therefore the action of $2\pi i(n_1\,\, n_2)$ in $\rmM_{1,2d}(\Z(1))$
on $\rmM_{1,2d}(\C)\times \rmM_{d,1}(\C)\times\C$ is by the translations
\[
2\pi i(n_1\,\, n_2)\colon (\alpha_1,\alpha_2,v,w)\mapsto
(\alpha_1+2\pi i n_1, \alpha_2+2\pi i n_2, v, w).
\]
Pushing this to the quotient gives
\begin{equation}
\label{eq:action_n}
2\pi i(n_1\,\, n_2)\colon (u,v,w)\mapsto 
(u+2\pi i n_1\tau+2\pi i n_2, v, w+2\pi i n_1v)\,.
\end{equation}
By inspection, one sees that the bijection in
Proposition~\ref{prop:paramD} is equivariant for the actions on its
source by $\rmM_{2d,1}(\Z)$, $\rmM_{1,2d}(\Z(1))$, and~$\Z(1)$ given
in (\ref{eq:action_m}) and (\ref{eq:action_n}) and the action on its 
target by $P^u(\Z)$ given in (\ref{eq:P-action}), where $2\pi i n$
in~$\Z(1)$, $(\substack{m_1\\m_2})$ in $\rmM_{2d,1}(\Z)$ and $2\pi
i(n_1\,\, n_2)$ in~$\rmM_{1,2d}(\Z(1))$ respectively correspond to
\begin{equation}\label{eq:equiv-actions}
\begin{pmatrix}
1 & 0 & 2\pi in\\
0 & 1_{2d} & 0\\
0 & 0 & 1    
\end{pmatrix}\,,\quad
\begin{pmatrix}
1 & 0 & 0 & 0\\
0 & 1_d & 0 & m_1\\
0 & 0 & 1_d & m_2\\
0 & 0 & 0 & 1    
\end{pmatrix}\,,\quad
\begin{pmatrix}
1 & 2\pi in_1 & 2\pi in_2 & 0\\
0 & 1_d & 0 & 0\\
0 & 0 & 1_d & 0\\
0 & 0 & 0 & 1    
\end{pmatrix}\,.
\end{equation}
This finishes our identification of $P^u(\Z)\backslash D$ with the
universal Poincar\'e torsor over the Siegel space~$\H_d$.
\end{proof}

\subsection{Duality and the Poincar\'e torsor}\label{sec4.7}
Proposition~\ref{prop:MSV_and_PT} together with the
equations~(\ref{eq:P-action}) give us an explicit description of the
Poincar\'e torsor over~$\H_d$.  Let $\tau$ be in~$\H_d$. Then we have
(as in Section~\ref{sec:ppavs})
$A_\tau=\rmM_{d,1}(\C)/ (\begin{matrix} 1_d &
  -\tau\end{matrix}){\cdot}\rmM_{2d,1}(\Z)$ (see the 2nd column of the
last matrix in~(\ref{eq:P-action}), and
$B_\tau=\rmM_{1,d}(\C)/\rmM_{1,2d}(\Z(1)){\cdot}(\substack{\tau\\1_d})$
(consider the first row), and the Poincar\'e torsor $\calP_\tau$ on
$A_\tau\times B_\tau$ that is the universal extension of $A_\tau$ by
$\C^\times$ and of $B_\tau$ by~$\C^\times$, giving isomorphisms
$B_\tau=\Ext^1(A_\tau,\C^\times)=A_\tau^\vee$ and
$A_\tau=\Ext^1(B_\tau,\C^\times)=B_\tau^\vee$.

Let now $f\colon B_\tau\to A_\tau$ be a morphism of abelian
varieties. Then $f$ is given by a complex linear map 
\[
\rmM_{1,d}(\C) \lto \rmM_{d,1}(\C),\quad u\mapsto f_\C{\cdot}u^t,
\quad\text{with $f_\C$ in $\rmM_d(\C)$}\,,
\]
and a $\Z$-linear map
\[
\rmM_{1,2d}(\Z(1)) \lto \rmM_{2d,1}(\Z),\quad 
2\pi i (\begin{matrix} n_1 & n_2\end{matrix})\mapsto 
f_\Z{\cdot}\begin{pmatrix} n_1^t\\ n_2^t\end{pmatrix}\,,
\]
with
\[
  f_\Z = 
\begin{pmatrix} \alpha & \beta \\\gamma & \delta\end{pmatrix} 
\in\rmM_{2d}(\Z)\,.
\]
The fact that these form a commutative diagram
\[
\begin{tikzcd}
\rmM_{1,2d}(\Z(1)) \ar[r,"{\cdot}(\substack{\tau\\ 1_d})"] 
\ar[d,"f_\Z"] & \rmM_{1,d}(\C) \ar[d,"f_\C"]\\
\rmM_{2d,1}(\Z)\ar[r, "(1_d\ -\tau){\cdot}"'] & \rmM_{d,1}(\C)
\end{tikzcd}
\]
gives us 
\begin{equation}\label{eq:fC_andf_Z}
2\pi i f_\C\tau^t = \alpha-\tau\gamma,\quad\text{and}\quad
2\pi i f_\C = \beta-\tau\delta\,.  
\end{equation}
The morphism $f\colon B_\tau\to A_\tau$ gives us the dual
$f^\vee\colon B_\tau\to A_\tau$. We want to know what $(f^\vee)_\C$
and $(f^\vee)_\Z$ are. The following proposition answers this
question.

\begin{Proposition}\label{prop:symm_on_Z}
In the situation above, we have
\[
  (f^\vee)_\Z = -(f_\Z)^t\,, \quad \text{and} \quad
  (f^\vee)_\C = \frac{1}{2\pi i}(-\gamma^t+\tau\delta^t)\,.
\]
\end{Proposition}
\begin{proof}
Let $b\in B_\tau$. By the rigidity of extensions of abelian varieties
by~$\Gm$,  $f^\vee(b)$ is the unique $a\in A_\tau$ such
that there is a morphism of extensions
\[
\begin{tikzcd}
\C^\times \ar[d, equal] \ar[r, hook] & \calP_{\tau,b} \ar[r, two heads] & A_\tau \\
\C^\times \ar[r, hook] & \calP_{\tau,a} \ar[r, two heads] \ar[u] &
B_\tau \ar[u, "f"']
\end{tikzcd}\,.
\] 
Let $u\in\rmM_{1,d}(\C)$ be an element that maps to~$b$. Then we are
looking for a $v$ in $\rmM_{d,1}(\C)$ (mapping to~$a$), $b_1$ and
$b_2$ in $\rmM_{d,1}(\Z)$, and $y$ in $\rmM_{d,1}(\C)$ such that the
diagram
\[
\begin{tikzcd}
(2\pi i n, 2\pi i(n_1\ n_2)) \ar[r, mapsto]
& \left(2\pi i (n + (n_1\ n_2){\cdot}(\substack{b_1\\ b_2})), 
f_\Z{\cdot}\left(\substack{n_1^t\\ n_2^t}\right)\right) \\
\Z(1) \oplus \rmM_{1,2d}(\Z(1)) \ar[r] 
\ar[dd,"{\cdot}\left(
\substack{1\  0\\ v\ \tau\\ 0\ 1_d}\right)"'] 
& \Z(1) \oplus \rmM_{2d,1}(\Z) 
\ar[dd, "\left(\substack{1\ 0\ -u\\ 0\ 1_d\ -\tau}\right){\cdot}"]\\ \\
\C \oplus \rmM_{1,d}(\C) \ar[r] & \C \oplus \rmM_{d,1}(\C) \\
(z,x) \ar[r,mapsto] & (z+x{\cdot}y, f_\C{\cdot}x^t)
\end{tikzcd}
\]
is commutative. This commutativity is equivalent to: for all
$n_1$ and $n_2$ in $\rmM_{1,d}(\Z)$
\[
2\pi i (n_1{\cdot}(v+\tau{\cdot}y) + n_2{\cdot}y) = 
2\pi i (n_1{\cdot}b_1+n_2{\cdot}b_2) -u{\cdot}(\gamma{\cdot} n_1^t + \delta{\cdot} n_2^t)\,,
\]
which in turn is equivalent to:
\[
2\pi i (v+\tau{\cdot}y) =  2\pi i b_1 -\gamma^t{\cdot}u^t 
\quad\text{and}\quad
2\pi i y = 2\pi i b_2 -\delta^t{\cdot}u^t\,.
\]
We solve this by taking
\[
b_1=0,\quad b_2=0, \quad 
y = -(2\pi i)^{-1}\delta^t{\cdot}u^t,\quad 
v = (2\pi i)^{-1} (-\gamma^t{\cdot}u^t + \tau{\cdot}\delta^t{\cdot}u^t)\,.
\]
We conclude that $f^\vee\colon B_\tau\to A_\tau$ is given by  
\[
\rmM_{1,d}(\C)\to\rmM_{d,1}(\C), \quad u\mapsto
(f^\vee)_\C{\cdot}u^t\,,
\]
with
\[
  (f^\vee)_\C = (2\pi i)^{-1} (-\gamma^t + \tau{\cdot}\delta^t)\,.
\]
The fact that $(f^\vee)_\Z$ is as claimed follows from the
commutativity of the diagram
\[
\begin{tikzcd}
2\pi i(n_1\ n_2) \ar[r, mapsto] \ar[ddd, mapsto, bend right=60] & 
-(\begin{smallmatrix} \alpha^t & \gamma^t\\ \beta^t &
  \delta^t\end{smallmatrix})
{\cdot}(\substack{n_1^t\\ n_2^t}) \ar[ddd, mapsto, bend left=60]\\
\rmM_{1,2d}(\Z(1)) \ar[r] \ar[d,"{\cdot}(\substack{\tau\\ 1_d})"'] & 
\rmM_{2d,1}(\Z) \ar[d,"(1_d\ -\tau){\cdot}"] \\
\rmM_{1,d}(\C) \ar[r] & \rmM_{d,1}(\C)\\
2\pi i (n_1{\cdot}\tau + n_2)\ar[r,mapsto] & 
(-\gamma^t + \tau{\cdot}\delta^t){\cdot}(\tau^t{\cdot}n_1^t+n_2^t)\,.
\end{tikzcd}
\]
To establish this commutativity one uses~(\ref{eq:fC_andf_Z}). 
\end{proof}
To finish this section, we include the polarisation 
\[
\Psi\colon \rmM_{2d,1}(\Z)\otimes\rmM_{2d,1}(\Z)\lto\Z(1), \quad
x\otimes y\mapsto 2\pi i\,x^t{\cdot}(\begin{smallmatrix} 0 & -1_d\\ 1_d &
  0\end{smallmatrix}){\cdot}y
\]
in the present discussion (up to here we have not used it, and the
results above are valid for $\tau$
in $\rmM_d(\C)$ whose imaginary part is invertible). Fixing the
second variable in $\Psi$ gives us the isomorphism
\[
\Psi_1\colon \rmM_{2d,1}(\Z) \lto \rmM_{2d,1}(\Z)^\vee(1),\quad
y\mapsto \left( x\mapsto \Psi(x\otimes y)\right)
\]
of $\Z$-Hodge structures (at~$\tau$ in~$\H_d$), and therefore an
isomorphism of complex tori
\begin{multline*}
 \lambda_\tau\colon 
 A_\tau = \rmM_{2d,1}(\C)/(M^{0,-1}_\tau+\rmM_{2d,1}(\Z)) \lto\\
 \lto
\rmM_{2d,1}(\C)^\vee/((M^\vee)^{0,-1}_\tau+\rmM_{2d,1}(\Z)^\vee(1))=B_\tau\,,
\end{multline*}
where the identification with $B_\tau$ is via universal
extensions as in the proof of Proposition~\ref{prop:MSV_and_PT}.
\begin{Proposition}\label{prop:pol_Z_and_C}
With the notation above, the $\C$-linear and $\Z$-linear maps
corresponding to~$\lambda_\tau$ are
\[
(\lambda_\tau)_\C\colon \rmM_{d,1}(\C)\to\rmM_{1,d}(\C)\,,\quad v\mapsto 2\pi i\, v^t
\]
and 
\[
\begin{aligned}
(\lambda_\tau)_\Z\colon \rmM_{2d,1}(\Z) & \to\rmM_{1,2d}(\Z)(1)\,,\\
y=(\substack{y_1\\ y_2}) & \mapsto 
2\pi i\, y^t{\cdot}(\begin{smallmatrix}0 & 1_d\\-1_d &
  0\end{smallmatrix}) = 2\pi i\, (-y_2^t\ \ y_1^t)\,.
\end{aligned}
\]
\end{Proposition}
\begin{proof}
For $(\lambda_\tau)_\Z$, this follows directly from the proof of
Proposition~\ref{prop:MSV_and_PT}. For $(\lambda_\tau)_\C$, it follows
from the commutativity of the diagram
\[
\begin{tikzcd}
\rmM_{2d,1}(\Z) \ar[r] \ar[d, "(1_d\ -\tau){\cdot}"'] & 
\rmM_{1,2d}(\Z)(1) \ar[d, "{\cdot}(\substack{\tau\\ 1_d})"] &
(\substack{y_1\\ y_2}) \ar[r, mapsto]& 2\pi i\, (-y_2^t \ \ y_1^t)\\
\rmM_{d,1}(\C) \ar[r] & \rmM_{1,d}(\C) & v\ar[r, mapsto] & 2\pi i\, v^t\,.
\end{tikzcd}
\]
Here one uses that $\tau^t=\tau$.
\end{proof}
It is reassuring to see, using Proposition~\ref{prop:symm_on_Z}, that , as
$(\lambda_\tau)_\Z=(\begin{smallmatrix}0 & 1_d\\-1_d &
  0\end{smallmatrix})$ is antisymmetric
$\lambda_\tau^\vee=\lambda_\tau$.

\section{Ribet varieties are special subvarieties}\label{sec5}
We recall that in Section~\ref{sec3} we had an abelian scheme $A\to S$
and a morphism $f\colon A^\vee\to A$, and $\alpha:=f-f^\vee\colon
A^\vee\to A$, hence $\alpha^\vee=-\alpha$, and a section $r_f$ of the
Poincar\'e torsor over the graph of~$\alpha$. Now we describe this in
the present context, over~$\C$, in the principally polarised case.

Let $M:=\Z(1)\oplus\Z^{2d}\oplus\Z$, $W$, $D$, and $P$ be as in
Section~\ref{sec:uptams}, and recall the notation $B_\tau$ from the
beginning of Section~\ref{sec4.7}. 
Let $\tau_0$ be in~$\H_d$, $f\colon B_{\tau_0}\to
A_{\tau_0}$ a morphism, and $\alpha:=f-f^\vee\colon B_{\tau_0}\to
A_{\tau_0}$. Then $\alpha$ gives (and is given by) the $\Z$-linear map
\begin{equation}\label{eq:alphaZ}
\begin{aligned}
\Gr^W_{-1}(M)^\vee(1) = \rmM_{1,2d}(\Z(1)) & \lto \rmM_{2d,1}(\Z) 
= \Gr^W_{-1} (M)\,,\\
2\pi i (\begin{matrix} n_1 & n_2\end{matrix}) & \mapsto 
\alpha_\Z{\cdot}\begin{pmatrix} n_1^t\\ n_2^t\end{pmatrix},
\end{aligned}
\end{equation}
with $\alpha_\Z \in\rmM_{2d}(\Z)$. 
By Proposition~\ref{prop:symm_on_Z},
\[
\alpha_\Z = f_\Z -(f^\vee)_\Z = f_\Z +(f_\Z)^t\,.
\]
Hence $\alpha_\Z$ is symmetric and the quadratic form
\[
\rmM_{1,2d}(\Z) \lto \Z,\quad 
x\mapsto \frac{1}{2}\,x{\cdot}\alpha_\Z{\cdot}x^t = x{\cdot}f_\Z{\cdot}x^t
\]
is $\Z$-valued. Just for completeness, we include that the
endomorphism $\beta:=\alpha\circ \lambda_{\tau_0}$ of $A_{\tau_0}$ is
anti-symmetric for the Rosati involution:
\[
\lambda_{\tau_0}^{-1}\circ \beta^\vee \circ \lambda_{\tau_0} = 
\lambda_{\tau_0}^{-1}\circ (\alpha\circ \lambda_{\tau_0})^\vee \circ
\lambda_{\tau_0} = 
\lambda_{\tau_0}^{-1}\circ \lambda_{\tau_0}^\vee\circ\alpha^\vee \circ
\lambda_{\tau_0} = -\alpha \circ \lambda_{\tau_0} =-\beta\,.
\]
Now, everything is in place to introduce the connected mixed Shimura
subvariety of the universal Poincar\'e-torsor $P^u(\Z)\backslash D$
over~$\H_d$ (quotiented by a suitable congruence subgroup
of~$\GSp(\Psi)(\Z)$) that is dictated by the map in~(\ref{eq:alphaZ})
being a morphism of Hodge structures. Concretely, we let $P_\alpha$
and $G_\alpha$ be the connected components of identity of the
stabilisers of~(\ref{eq:alphaZ}) in~$P$ and in~$\GSp(\Psi)$. As the
action of $P$ on $\Gr^W_{-1}(M)$ factors through $\GSp(\Psi)$,
$P_\alpha$ is the inverse image in~$P$ of~$G_\alpha$, and the
unipotent radical $P_\alpha^u$ of $P_\alpha$ is equal to~$P^u$, hence
contains~$U$.  In $D$ and $\H_d$, we consider the orbits
\begin{equation}\label{eq:Dalpha}
D_\alpha:=P_\alpha(\R)^+U(\C){\cdot}\widetilde{\tau_0} \subset D
\quad\text{and}\quad
\H_{d,\alpha}:=G_\alpha(\R)^+{\cdot}\tau_0 \subset \H_d \,,
\end{equation}
where $\widetilde{\tau_0}$ is the element of $D$ that corresponds to
$(\tau_0,0,0,0)$ under the bijection of
Proposition~\ref{prop:paramD}. More intrinsically:
$\widetilde{\tau_0}$ is the mixed Hodge structure on $M$ in which the
weight filtration is split over $\Z$ by the given $\Z$-basis, and
which induces that given by $\tau_0$ on~$\Gr^W_{-1}M$. Here, it does
not matter which lift of $\tau_0$ we take, but it will matter further
on when we describe the Ribet section in~$D_\alpha$.

Deligne's group theoretical description of Shimura varieties shows
that $\H_{d,\alpha}$ is the connected component containing $\tau_0$ of
the set of $\tau\in\H_d$ where~(\ref{eq:alphaZ}) is a morphism of
Hodge structures (equivalently: where it induces a morphism
$\alpha\colon B_\tau\to A_\tau$). Let us explain in a few lines how
this works; for details, see~\cite[Section~2.4]{Moonen-L1}
and~\cite[Section~1.1.12]{Deligne-VS}. Pure Hodge structures on an
$\R$-vector space correspond to $\R$-algebraic actions
of~$\C^\times$. For $G$ a connected linear algebraic group over~$\R$,
the set of $\R$-morphisms $\Hom(\C^\times,G(\R))$ is the set of
$\R$-points of a smooth $\R$-scheme, which is the disjoint union of
$G$-orbits (for $G$ acting by composition with inner automorphisms).
The $G(\R)^+$-orbits in $\Hom(\C^\times,G(\R))$ are the connected
components for the Archimedean topology. References in~\cite{SGA3.II}
(SGA~3): Exp.~IX, Cor.~3.3, and Exp.~XI, Cor.~4.2.

The pairs $(P_\Q,D)$, $(G_{\alpha,\Q},\H_{d,\alpha})$ and
$(P_{\alpha,\Q},D_\alpha)$ are connected mixed Shimura data as
in~\cite[Def.~2.1]{Pink1}, and we have the diagram of morphisms of
connected mixed Shimura data 
\begin{equation}
\begin{tikzcd}
(P_{\alpha,\Q},D_\alpha) \ar[r, hook] \ar[d, two heads]& (P_\Q,D) \ar[d, two heads]\\
(G_{\alpha,\Q},\H_{d,\alpha}) \ar[r,hook] & (\GSp(\Psi)_\Q,\H_d)\,.
\end{tikzcd}  
\end{equation}
The careful reader will have noticed that we must show that $D$ is a
$P(\R)^+U(\C)$-orbit in $\Hom(\C^\times\times\C^\times,P(\C))$ and
$D_\alpha$ is a $P_\alpha(\R)^+U(\C)$-orbit in
$\Hom(\C^\times\times\C^\times,P_\alpha(\C))$.  For the fact that the
natural maps from these orbits to $D$ and $D_\alpha$ are isomorphisms
we refer to Propositions~1.18 and~1.16(c) in~\cite{Pink_thesis} (the
surjectivity is clear because source and target are orbits for the
same group, for the injectivity one has to show that the stabilisers
are the same).

\begin{Proposition}\label{prop:Poinc_torsor_over_H_alpha}
The quotient $P_\alpha^u(\Z)\backslash D_\alpha$ is the universal
Poincar\'e torsor over~$\H_{d,\alpha}$. The quotient of $D_\alpha$ by
$P_\alpha^u(\Z)U(\C)$ is the universal family of $A_\tau\times B_\tau$'s
over~$\H_{d,\alpha}$. The quotient of $D_\alpha$ by
$P_\alpha^u(\Z)\rmM_{1,2d}(\R)U(\C)$
is the universal family of
$A_\tau$'s over~$\H_{d,\alpha}$, and the quotient of $D_\alpha$ by
$P_\alpha^u(\Z)\rmM_{2d,1}(\R)U(\C)$ is the universal family of
$B_\tau$'s over~$\H_{d,\alpha}$.
\end{Proposition}
\begin{proof}
One easily deduces this from Proposition~\ref{prop:MSV_and_PT} and
parts of its proof. 
\end{proof}
Now we proceed directly to the Ribet section, by revealing the tensor
that defines it, namely, the map (encoded by a
matrix~$\widetilde{\alpha_\Z}$) 
\begin{equation}\label{eq:map_tilde_alphaZ}
\begin{tikzcd}
M^\vee(1) = \Z\oplus\rmM_{1,2d}(\Z(1))\oplus\Z(1) \ar[r] &
\Z(1)\oplus\rmM_{2d,1}(\Z)\oplus\Z = M \\
x =  (\begin{matrix} k_1 & 2\pi i n & 2\pi i k_2\end{matrix}) 
\ar[r,mapsto] & 
\begin{pmatrix} -2\pi i k_2\\ \alpha_\Z{\cdot}n^t \\
 -k_1\end{pmatrix} = \widetilde{\alpha_\Z} {\cdot} x^t \,, 
\end{tikzcd}
\end{equation}
where
\begin{equation}\label{eq:tilde_alphaZ}
\widetilde{\alpha_\Z} = 
\begin{pmatrix}
0 & 0 & -1 \\
0 & \alpha_\Z & 0 \\
-1 & 0 & 0
\end{pmatrix}\quad \text{in $\rmM_{2d+2}(\Z)$}\,.
\end{equation}
This tensor was already described in~\cite{Ribet}, see
also~\cite[Lemme~6]{Bertrand13a}.
We let $P_{\tilde{\alpha}}$ be the stabiliser in $P$ of this
map~(\ref{eq:map_tilde_alphaZ}), as a group scheme over~$\Z$. Then,
for any $\Z$-algebra $R$ and for any $p$ in $P(R)$ we have $p\in
P_{\tilde{\alpha}}(R)$ if and only if $p{\cdot}\widetilde{\alpha_\Z} =
\mu(p)\widetilde{\alpha_\Z}{\cdot}p^{-1,t}$ in $\rmM_{2d+2}(R)$, which
is equivalent to $p{\cdot}\widetilde{\alpha_\Z}{\cdot}p^t
=\mu(p)\widetilde{\alpha_\Z}$. A direct computation then shows, for
any $\Z$-algebra $R$ in which multiplication by $2$ is injective:
\begin{equation}
\label{eq:Ptildealpha}
P_{\tilde{\alpha}}(R) = 
\left\{
\begin{pmatrix}
\mu(g) & x & \mu(g)^{-1}x f_\Z x^t \\
0 & g & \mu(g)^{-1}g\alpha_\Z x^t \\
0 & 0 & 1    
\end{pmatrix} :
\text{$(g,\mu(g))\in G_\alpha(R)$, $x\in\rmM_{1,2d}(R)$}
\right\},
\end{equation}
where the matrices are with respect to the $\Z$-basis $2\pi ie_0,
e_1,\ldots,e_{2d+1}$ of~$M$. We note that for $R$ on which
multiplication by $2$ is injective, $P_{\tilde{\alpha}}(R)$ is the
semi-direct product
\begin{equation}\label{eq:P_alpha_semidir_prod}
P_{\tilde{\alpha}}(R) = \rmM_{1,2d}(R) \rtimes G_\alpha(R) = 
\left\{
\begin{pmatrix}
1 & x & x f_\Z x^t \\
0 & 1_{2d} & \alpha_\Z x^t \\
0 & 0 & 1    
\end{pmatrix}
\right\}\cdot
\left\{
\begin{pmatrix}
\mu(g) & 0 & 0 \\
0 & g & 0 \\
0 & 0 & 1    
\end{pmatrix}
\right\} \,.
\end{equation}
where $x$ ranges over $\rmM_{1,2d}(R)$ and $g$ over~$G_\alpha(R)$. In
particular, the unipotent radical (over $\Z[1/2]$) of
$P_{\tilde{\alpha}}$ is a vector group scheme, and the weight $-2$
part of its Lie algebra is zero. We define
\begin{equation}
D_{\tilde{\alpha}} :=
P_{\tilde{\alpha}}(\R)^+{\cdot}\widetilde{\tau_0} \subset
D_\alpha\subset D \,.  
\end{equation}
Then we have the following diagram of connected mixed Shimura data
\begin{equation}
\begin{tikzcd}
(P_{\widetilde{\alpha},\Q},D_{\widetilde{\alpha}}) \ar[r, hook] \ar[rd,
two heads] & 
(P_{\alpha,\Q},D_\alpha) \ar[d, two heads]\\
& (G_{\alpha,\Q},\H_{d,\alpha}) \,.
\end{tikzcd}
\end{equation}

\begin{Theorem}\label{thm:RibetsectionMSV}
  The quotient $P_{\tilde{\alpha}}^u(\Z)\backslash D_{\tilde{\alpha}}$
  is the image of a section $r^\Sh_f$ in $P_\alpha^u(\Z)\backslash
  D_\alpha$ (the universal Poincar\'e torsor over~$\H_{d,\alpha}$, see
  Proposition~\ref{prop:Poinc_torsor_over_H_alpha}) over the family of
  $B_\tau$ with $\tau$ ranging over~$\H_{d,\alpha}$. In particular,
  the image of $r^\Sh_f$ is a special subvariety. This section $r^\Sh_f$
  is equal, in this setting, to the section $r_f$ of
  Proposition~\ref{prop:ribsec}. 
\end{Theorem}
\begin{proof}
  It is sufficient to verify the claim at each
  $\tau\in\H_{d,\alpha}$. So let $\tau$ be such. The description
  in~(\ref{eq:P-action}) of the action of
  $P^u_\alpha(\R)U(\C)=P^u(\R)U(\C)$ on $D$ shows that it is free and
  transitive on the fibre $D_{\alpha,\tau}$ of
  $D_\alpha\to\H_{d,\alpha}$ at~$\tau$. This gives us the bijection
\begin{equation}\label{eq:D_alpha_tau}
\begin{tikzcd}
P^u_\alpha(\R)U(\C) \ar[r,"\simeq"] & D_{\alpha,\tau}
\end{tikzcd}\,\quad
p\mapsto p{\cdot}\widetilde{\tau}\,,  
\end{equation}
where $\widetilde{\tau}$ is the element of $D$ that corresponds to
$(\tau,0,0,0)$ under the bijection of Proposition~\ref{prop:paramD}.
For $g$ in $G_\alpha(\R)^+$ with $g{\cdot}\tau_0=\tau$, we have $g\in
P_{\widetilde{\alpha}}(\R)^+$ via~(\ref{eq:P_alpha_semidir_prod}), and
$\widetilde{\tau}=g{\cdot}\widetilde{\tau_0}$ (use
(\ref{eq:GSp-action})), hence
$D_{\widetilde{\alpha},\tau}=P^u_{\widetilde{\alpha}}(\R){\cdot}\widetilde{\tau}$.
Via the bijection~(\ref{eq:D_alpha_tau}), the inclusion
$D_{\widetilde{\alpha},\tau}\subset D_{\alpha,\tau}$ corresponds to
the inclusion $P^u_{\widetilde{\alpha}}(\R)\subset
P^u_\alpha(\R)U(\C)$, and the $P^u_{\widetilde{\alpha}}(\Z)$-action on
$D_{\tilde{\alpha}}$ corresponds to the action by left-multiplication
on $P^u_{\widetilde{\alpha}}(\R)$. By~(\ref{eq:P_alpha_semidir_prod}),
$P^u_{\widetilde{\alpha}}(\R)=\rmM_{1,2d}(\R)$, and
(\ref{eq:D_alpha_tau}) identifies this with $\rmM_{1,d}(\C)$, sending
$(x_1,x_2)$ to $2\pi i(x_1\tau + x_2)$ by~(\ref{eq:P-action}). Hence
$P_{\widetilde{\alpha}}^u(\Z)\backslash
D_{\widetilde{\alpha},\tau}=B_\tau$.
Proposition~\ref{prop:Poinc_torsor_over_H_alpha} together with the
description~(\ref{eq:P_alpha_semidir_prod}) of
$P_{\widetilde{\alpha}}^u$ show that
$P_{\tilde{\alpha}}^u(\Z)\backslash D_{\tilde{\alpha}}$ is the image
of a section $r^\Sh_f$ of the Poincar\'e torsor over the graph of
$\alpha\colon B_\tau\to A_\tau$ (equivalently, over~$B_\tau$). This
section differs from $r_f$ by multiplication by a global regular
function on~$B_\tau$, hence by a constant factor in~$\C^\times$. As
both sections have value~$1$ at $0\in B_\tau$, they are equal.
\end{proof}

\begin{Proposition}\label{prop_tors_prop_r_f_msv}
  Let $\tau$ be an element of~$\H_{d,\alpha}$. The subset
  $\rmM_{1,2d}(\Z)\backslash\rmM_{1,2d}(\R){\cdot}1$ of
  $P_\alpha^u(\Z)\backslash P_\alpha^u(\R)U(\C)$ corresponds, under
  the bijection in~(\ref{eq:D_alpha_tau}), to the unit section over
  $B_\tau$ of the Poincar\'e torsor $P_\alpha^u(\Z)\backslash
  D_{\alpha,\tau}$ on $A_\tau\times B_\tau$ (see
  Proposition~\ref{prop:Poinc_torsor_over_H_alpha}). 
  For $x$ in $\rmM_{1,2d}(\R)$ and
  $\ol{x}$ its image in~$B_\tau$, the extension $E_{\tau,\ol{x}}$ of
  $A_\tau$ by $\C^\times$ corresponding to~$\ol{x}$ is,
  as real Lie group, $(\C/2\pi i\Z)\times(\rmM_{2d,1}(\R)/\rmM_{2d,1}(\Z))$,
  and $r_f(\ol{x})$ is given by $(2\pi ixf_\Z x^t, \alpha_\Z x^t)$. If $\ol{x}$ is of
  order $n$ in~$B_\tau$, then $r_f(\ol{x})$ in $E_{\tau,\ol{x}}$ is killed by~$n^2$.
\end{Proposition}
\begin{proof}
Consider~(\ref{eq:D_alpha_tau}) and~(\ref{eq:P-action}). 
Let $x=(x_1,x_2)\in\rmM_{1,2d}(\R)$. This gives the elements
\[
p_x := 
\begin{pmatrix}
1 & 2\pi ix_1 & 2\pi ix_2 & 0\\
0 & 1_d & 0 & 0\\
0 & 0 & 1_d & 0\\
0 & 0 & 0 & 1    
\end{pmatrix}
\in P^u_\alpha(\R)\,,
\]
and
\[
p_x{\cdot}\tilde{\tau} = 
\begin{pmatrix}
2\pi i(x_1\tau + x_2)& 0\\ 
\tau & 0\\ 1_d & 0\\ 0 & 1
\end{pmatrix}
\C^{d+1}\in D_{\alpha,\tau}\,.
\]
This proves the first claim of the proposition. To describe
$E_{\tau,\ol{x}}$, let, for $z$ in~$\C$ and
$(\begin{smallmatrix}y_1\\y_2\end{smallmatrix})$ in~$\rmM_{2d,1}(\R)$,
\[
p_{z,y} :=
\begin{pmatrix}
1 & 0 & 0 & 2\pi i z\\
0 & 1_d & 0 & y_1\\
0 & 0 & 1_d & y_2\\
0 & 0 & 0 & 1    
\end{pmatrix}\,,
\]
and then
\[
p_{z,y}{\cdot}p_x{\cdot}\tilde{\tau} = 
\begin{pmatrix}
2\pi i(x_1\tau + x_2)& 2\pi i(z-(x_1\tau+x_2)y_2)\\ 
\tau & y_1-\tau y_2\\ 1_d & 0\\ 0 & 1
\end{pmatrix}
\C^{d+1}\,.
\]
Now observe that $2\pi i(z-(x_1\tau+x_2)y_2)$ and $y_1-\tau y_2$ are
$\R$-linear in~$z$, $y_1$ and~$y_2$, and that $2\pi i(x_1\tau + x_2)$
does not depend on~$z$, $y_1$ and~$y_2$. Hence the $\R$-vector space
structure on $\{2\pi i(x_1\tau + x_2)\}\times\rmM_{d,1}(\C)\times \C$
in~$D_{\alpha,\tau}$ corresponds to the $\R$-vector space structure on
$\rmM_{2d,1}(\R)\times\C$ on the left, and therefore the same holds
for the group structures. The left-action by the $p_{z,y}$ with
$z\in\Z$ and $y\in\rmM_{2d,1}(\Z)$ on these 2 real vector spaces then
gives the description of~$E_{\tau,\ol{x}}$. The description
of~$P^u_{\tilde{\alpha}}$ in~(\ref{eq:P_alpha_semidir_prod}) proves
the last two claims in the proposition.
\end{proof}

\begin{Remark}\label{rem_r_f_msv}
Assume that $\alpha$ is an isogeny. 
\begin{enumerate}
\item The tensor $\widetilde{\alpha}$ in~(\ref{eq:map_tilde_alphaZ})
  that defines the Ribet variety as an irreducible component of its
  Hodge locus is a selfduality of mixed $\Q$-Hodge structures. It is
  interesting to see that on the underlying $\Z$-module $M$ it is a
  symmetric $\Z(1)$-valued pairing. Algebraically this can be
  described as a self-duality of $1$-motives with $\Q$-coefficients,
  see~\cite{Ribet} and also~\cite{Bertrand13a}.
\item\label{rem_r_f_msv_item_2} Let $\Gamma_\alpha(3)$ be the kernel
  of $G_\alpha(\Z)\to G_\alpha(\F_3)$. Then $\Gamma_\alpha(3)$ acts on
  the whole situation of Theorem~\ref{thm:RibetsectionMSV}, and freely
  on~$\H_{d,\alpha}$. The quotient $\Gamma_\alpha(3) \backslash 
  D_\alpha$ is then the Poincar\'e torsor $\calP$ for the abelian
  scheme $A:=\Gamma_\alpha(3)\backslash
  (P_\alpha^u(\Z)\rmM_{1,2d}(\R)U(\C)\backslash D_\alpha)$ over the
  pure Shimura variety $S:=\Gamma_{\alpha}(3)\backslash
  \H_{d,\alpha}$, with the image of the Ribet section~$r_f$ as a
  special subvariety of a family of semi-abelian varieties. As a
  generalisation of Lemma~\ref{lemma:tgeneric}, we will now prove that
  this special subvariety is not a torsion translate of a family of
  algebraic subgroups. Let $\tau$ be in $\H_{d,\alpha}$ and
  $x=(x_1,\ldots,x_{2d})$ be in $\rmM_{1,2d}(\R)$ such that
  $x_1,\ldots,x_{2d}, x\alpha_\Z x^t$ in~$\R$ are $\Q$-linearly
  independent. Then the coordinates of $\alpha_\Z{\cdot}x^t$ and
  $x\alpha_\Z x^t$ are $\Q$-linearly independent. By
  Proposition~\ref{prop_tors_prop_r_f_msv}, the subgroup of
  $E_{\tau,\ol{x}}$ generated by $r_f(\ol{x})$ is dense, for the
  Archimedean topology, in $(i\R/2\pi
  i\Z)\times(\rmM_{2d,1}(\R)/\rmM_{2d,1}(\Z))$. This shows that the
  union of the images of the $nr_f$, with $n\in\Z$, is dense, for the
  Archimedean topology, in a circle bundle of real codimension~1
  in~$\cal{P}$. The fibres of this circle bundle are the maximal
  compact subgroups of the corresponding complex analytic semi-abelian
  varieties.
\item The example just given (the image of~$r_f$) now supports Pink's
  Conjecture~1.3 of ~\cite{Pink2}: indeed, it is a subvariety $Y$ of
  $\calP$ containing a Zariski dense set of special points
  (i.e.~special subvarieties of maximal codimension in~$\calP$), and
  it is itself a special subvariety of~$\calP$. For further
  verifications in this context of~\cite{Pink2}, Conjecture~1.3,
  see~\cite{BMPZ} and~\cite{Bertrand18}. 
\item\label{rem_r_f_msv_item_4} Let us now clarify what is wrong in
  the proof of Theorem~6.3 of~\cite{Pink2}. The error is in the
  statement ``Since the special subvarieties of $A$ that dominate $S$
  are precisely the translates of semiabelian subschemes by torsion
  points,\ldots''; we have just seen that this is not true. Similarly,
  note the sentence ``Conversely, for any special subvariety $T\subset
  A$, every irreducible component of $T\cap A_s$ is a translate of a
  semiabelian subvariety of $A_s$ by a torsion point.'' in the proof
  of Theorem~5.7 of~\cite{Pink2}.

  The essential difference between the case of Kuga varieties (Shimura
  families of abelian varieties over pure Shimura varieties), where
  the statement is correct (\cite{Pink1}, Proposition~4.6), and the
  case of Shimura families of tori over Kuga varieties is as
  follows. In the first case the morphism of mixed Shimura varieties
  $A\to S$ is induced by a morphism of Shimura data $(P,D_P)\to
  (G,D_G)$ with $G$ reductive, and $P\to G$ surjective, split, with
  kernel $V$ a $\Q$-vector space. Then the special subvarieties $Z$ of
  $A$ that surject to~$S$ are given by morphisms of sub-Shimura data
  $(Q,D_Q)$ of $(P,D_P)$, with $Q\to G$ is surjective. Then $Q$ is an
  extension of $G$ by $Q\cap V$, a sub-$\Q$-vector space of~$V$. This
  extension is split because $\rmH^2(G,Q\cap V)=0$, and the splitting
  is unique up to conjugation by $Q\cap V$ because $\rmH^1(G,Q\cap
  V)=0$. So indeed such special subvarieties come from subfamilies
  $B\to S$ of $A\to S$ and Hecke correspondences that account for
  translations by torsion points. In the second case, say $T\to A$,
  these arguments no longer apply because the group $P$ in the Shimura
  datum for $A$ (such as $P_\alpha/U$ as above) is not necessarily
  reductive (and indeed the extension $P_\alpha$ of $P_\alpha/U$ by
  $U$ is not split).
\end{enumerate}
\end{Remark}

\section{The elliptic curve example, via generalised
  jacobians}\label{sec6}  

In this section we give a description of the example in
Section~\ref{sec2} in terms of the generalised jacobian of a family of
singular curves. Our reason to include it is that this description is
more elementary than the one using the Poincar\'e bundle, and that it
is more explicit in terms of divisors, rational functions, Weil
pairing, and is a nice application of Weil reciprocity.

We return to the situation as in Section~\ref{sec2}, except that now
we let $k$ be an arbitrary algebraically closed field. Let $E$ be an
elliptic curve over~$k$. Here we will view $E\times E$ as a family of
elliptic curves over $E$ via the 2nd projection $\pr_2\colon E_E =
E\times E\to E$, $(x,y)\mapsto y$.

In our construction, we will remove a finite number of points of the
base curve~$E$, and denote the complement by~$U$. This $U$ will be
shrunk a few times.

The diagonal morphism $\Delta\colon E\to E_E$, $x\mapsto(x,x)$, is a
section, and the group law of $E_E$ over $E$ gives us a second
section~$2\Delta$, $x\mapsto (2x,x)$.  The sections $\Delta$ and
$2\Delta$ are disjoint over the open subset $U:=E-\{0\}$.

We let $C\to U$ be the singular curve over $U$ obtained by identifying
the disjoint sections $2\Delta$ and $\Delta$. As a set, it is the
quotient of $E_U$ by the equivalence relation generated by
$(2x,x)\sim (x,x)$ with $x$ ranging over~$U$. The topology on $C$
is the finest one for which the quotient map $\quot\colon E_U\to C$
is continuous: a subset $V$ of $C$ is open if and only if
$\quot^{-1}V$ is open in~$E_U$. The regular functions on an open
set $V$ of $C$ are the regular functions $f$ on $\quot^{-1}V$ such
that $f(2x,x)=f(x,x)$ whenever $\quot(x,x)$ is in~$V$. It is
proved in Theorem~5.4 of~\cite{Ferrand} that this topological space with sheaf
of rings is indeed an algebraic variety over~$k$.  In the category of
varieties over~$k$, $\quot\colon E_U\to C$ is the co-equaliser of the pair of
morphisms $(2\Delta,\Delta)$ from $U$ to~$E_U$:
\[
\begin{tikzcd}
U \ar[r,shift left,"2\Delta"] \ar[r,shift right,"\Delta"']& E_U \ar[r,"\quot"]& C
\end{tikzcd}\,.
\]

The curve $C\to U$ is a family of singular curves, each with an
ordinary double point; it is semi-stable of genus~2
(see~\cite[9.2/6, 9.2/8]{BLR}). Its normalisation is
$\quot\colon E_U \to C$. Its generalised jacobian 
\[
G:=\Pic^0_{C/U}
\]
is described in \cite{BLR}, 8.1/4, 8.2/7, 9.2, 9.4/1, and
in more direct terms in this specific situation in~\cite{Howe}. As
$C\to U$ has a section (for example $\ol{\Delta}:=\quot\circ \Delta$),
we have, for every $T\to U$, that $G(T)$ is equal to $\Pic^0(C_T/T)/\Pic(T)$, where
$\Pic^0(C_T/T)$ is the group of isomorphism classes of invertible
$\calO$-modules on $C_T$ that have degree zero on the fibres
of $C_T\to T$. The group $\Pic(T)$ is contained as direct
summand in $\Pic^0(C_T/T)$ via pullback by the projection
$C_T\to T$ and a chosen section. In particular, a divisor $D$
on $C$ that is finite over~$U$, disjoint from $\ol{\Delta}(U)$ and of
degree zero after restriction to the fibres of $C\to U$ gives the
invertible $\calO_C$-module $\calO_C(D)$ that has degree zero on the
fibres and therefore gives an element denoted $[D]$ in~$G(U)$. An
alternative and very useful description, given in detail
in~\cite{Howe}, of $\Pic(C_T)$ is the set of isomorphism
classes of $(\calL,\sigma)$, with $\calL$ an invertible $\calO$-module
on $E_T$ and $\sigma\colon (2\Delta)^*\calL\to \Delta^*\calL$
an isomorphism of $\calO$-modules on~$T$, where an isomorphism from
$(\calL,\sigma)$ to $(\calL',\sigma')$ is an isomorphism
$f\colon\calL\to\calL'$ such that $(\Delta^*f)\circ\sigma =
\sigma'\circ (2\Delta)^*f$.

For $x$ in $U$, the fibre $G_x$ is, as abelian group, the
group~$\Pic^0(C_x)$. In terms of divisors this is the quotient of the
group $\Div^0(C_x)$ of degree zero divisors with support outside
$\{\ol{\Delta}(x)\}$ by the subgroup of principal divisors
$\divisor(f)$ for nonzero rational functions $f$ in $k(C_x)^\times$
that are regular and invertible at~$\ol{\Delta}(x)$. As
$C_x-\{\ol{\Delta}(x)\}$ is the same as $E-\{2x,x\}$, $\Div^0(C_x)$ is
the group of degree zero divisors on $E$ with support outside
$\{2x,x\}$. An element $f$ of $k(C_x)^\times$ that is regular at
$\ol{\Delta}(x)$ is an element of $k(E)^\times$ that is regular at
$2x$ and $x$ and satisfies $f(2x)=f(x)$. This gives us a useful
description of~$G_x$.

The normalisation map $\quot\colon E_U\to C$ induces a morphism of
group schemes over~$U$
\[
\pi\colon  G = \Pic^0_{C/U} \rightarrow \Pic^0_{E_U/U} =  E_U,
\]
and identifies $G$ with the extension of $E$ by $\Gm$ given by the
section $\Delta\in E_U(U)$. For $x$ in $U$ and $D \in
\Div^0(C_x)$, the class $[D]$ in $G_x$ lies in the kernel
$k^\times$ of $\pi_x$ if and only if there exists $f \in
k(E)^\times$ such that $D=\divisor(f)$ on $E$, and it
is then a torsion point in $k^\times$ if and only if the quotient
$f(2x)/f(x) \in k^\times$, which does not depend on the
choice of~$f$, is a root of unity.

We recall that for $u$ in $\End(E)$, the pullback map $u^*$ on
$\Div(E)$ induces $u^\vee$ in $\End(E^\vee)$, the dual of~$u$, and
then $\ol{u}:=\lambda^{-1} u^\vee\lambda$ in~$\End(E)$ is called the
Rosati-dual of~$u$, where $\lambda$ is the standard polarisation as in
Section~\ref{sec2}.  The map $\End(E)\to\End(E)$, $u\mapsto \ol{u}$ is
a anti-morphism of rings, in fact an involution. It is characterised
by the property that in $\End(E)$ we have
$\ol{u}u=\deg(u)=\deg(\ol{u})$ and $u+\ol{u}\in\Z$. Also, the
pushforward map $u_*$ on $\Div(E)$ induces an element still denoted
$u_*$ in $\End(E^\vee)$ such that $\lambda u=u_*\lambda$ in $\Hom(E,
E^\vee)$, and $u_*u^*=\deg(u)$ in~$\End(E^\vee)$. Hence $u_*$ and
$u^*$ are each other's Rosati duals. For $f$ a nonzero rational function
on $E$ and $u\neq 0$ we have $u^*\divisor(f)=\divisor(f\circ u)$, and
$u_*\divisor(f)=\divisor(\Norm_u(f))$, where $\Norm_u\colon
k(E)^\times\to k(E)^\times$ is the norm map along~$u$.

We will use Weil reciprocity: for $f$ and $g$ nonzero rational
functions on~$E$ such that $\divisor(f)$ and $\divisor(g)$ have
disjoint supports, one has $f(\divisor(g))=g(\divisor(f))$, where for
$D=\sum_PD(P){\cdot}P$ a divisor on~$E$ one defines
$f(D)=\prod_Pf(P)^{D(P)}$, cf.~\cite{Serre}, III, Proposition~7.  

We will also use the Weil pairing. For $n$ a positive integer and $P$
and $Q$ in $E[n]$ the element $e_n(P,Q)$ in $\mu_n(k)$ is defined as
follows. Let $D_P$ and $D_Q$ in $\Div^0(E)$ be disjoint divisors
representing $\lambda(P)$ and~$\lambda(Q)$. Let $f$ and $g$ be in
$k(E)^\times$ such that $nD_P=\divisor(f)$ and
$nD_Q=\divisor(g)$. Then $e_n(P,Q)=f(D_Q)/g(D_P)$. For $n$ invertible
in $k$ this pairing $e_n$ is a perfect alternating pairing, see
\cite{Husemoller}, Chapter~12, Remark~3.7.

We assume that $\phi$ is an endomorphism of $E$ such that
$\alpha:=\phi-\ol{\phi}\neq 0$. We set
\begin{equation}
  \label{eq:def_D}
D_\phi := \phi_*\big((\Delta)-(2\Delta)\big) - 
\phi^*\big((\Delta)-(2\Delta)\big) 
\quad\text{in $\Div^0(E_U)$}. 
\end{equation}
Note that $(\Delta)-(2\Delta)$ is linearly equivalent to
$(0)-(\Delta)$, and that, under $\lambda\colon E\to E^\vee$, $\Delta$
in $E(E)$ is mapped to $[(0)-(\Delta)]$.
We want the support of $D_\phi$ to be disjoint from $\Delta$ and
$2\Delta$, and this becomes true by removing from $U$ the kernels of
$2(\phi-1)$, of $2\phi-1$ and of $\phi-2$ (as $\ol{\phi}\neq\phi$,
only a finite set is removed). We can now also view $D_\phi$ as
element of $\Div^0(C)$, and we set:
\begin{equation}
\label{eq:def_sJ}
t^J_\phi := [D_\phi] \quad\text{in $G(U)$.}
\end{equation}
Combining Parts~2 and~4 of the following theorem
provides a new proof in the elliptic case of Proposition~\ref{prop:ribsectors},
while Part~3 sharpens Theorem~\ref{thm:ttorsionproperty}.
\begin{Theorem}\label{thm:tJ}
\begin{enumerate}
\item The image $\pi(t^J_\phi)$ of $t^J_\phi$ equals
\[
  (\alpha,\id_U)\colon U\to E\times U = E_U\,.
\]
\item Let $n$ be a positive integer and $x$ in $E$ with $nx=0$. 
Then $n^2t^J_\phi(x)=0$ in~$G_x$, and $nt^J_\phi(x) = e_n(\phi(x),x)$.
\item Let $n$ be a positive odd integer that is prime to
  $\deg(\alpha)$, invertible in~$k$, and that divides none among
  $\deg(2(\phi-1))$, $\deg(2\phi-1)$ and $\deg(\phi-2)$. Then there is
  an $x\in U$ of order~$n$, such that the order of $t^J_\phi(x)$ is
  equal to~$n^2$.
\item The extension $G$ of $E_U$ by ${\Gm}_U$ is uniquely isomorphic
  to the restriction to $E_U$ of the Poincar\'e torsor $\calP$ as in
  Section~\ref{sec2} (up to a switch of the factors of $E \times E$),
  and under this isomorphism, $t^J_\phi$ equals the Ribet
  section~$t_\phi$.
\end{enumerate}
\end{Theorem}
\begin{proof}
  We prove part~1.  The image $\pi(t^J_\phi)$ in $E_U(U)$ of
  $t^J_\phi$ is the class of the divisor $D_\phi$ on~$E_U$, hence we
  have, denoting by $\simeq$ linear equivalence on $\Div^0(E_U)$:
\begin{align*}
D_\phi & \simeq \phi_*\big((\Delta)-(2\Delta)\big) - \ol{\phi}_*\big((\Delta)-(2\Delta)\big)  \\
& =  \big((\phi(\Delta))-(2\phi(\Delta))\big) - \big((\ol{\phi}(\Delta))-(2\ol{\phi}(\Delta))\big) \\
& \simeq \big( (0) - (\phi(\Delta))\big) - \big((0) - (\ol{\phi}(\Delta))\big) \\
& \simeq \big((0) - ((\phi-\ol{\phi})(\Delta))\big) = \big((0) - (\alpha(\Delta))\big).
\end{align*}
Under the principal polarisation $\lambda\colon E\to E^\vee$,
$x\mapsto [(0)-(x)]$, this corresponds to $\alpha(\Delta)$ in~$E(U)$.
This proof of part~1 is finished.

We prove part~2. So, let $n$ be a positive integer, and let $x\in U$
be a point such that $nx=0$ in~$E$. As $nx=0$, we have $n\pi
t^J_\phi(x)=n \alpha(x)=0$ in~$E$. This means that $nD_{\phi,x}$ is a
principal divisor on~$E$. Let $f\in k(E)^\times$ be such that
$\divisor(f)=n(x)-n(2x)$ in $\Div(E)$. Then we have, on~$E$:
\begin{align*}
\divisor(f\circ \phi) & = \phi^*\divisor(f) = \phi^*\left(n(x)-n(2x)\right),\\
\divisor(\Norm_\phi(f)) & = \phi_*\divisor(f) = \phi_*\left(n(x)-n(2x)\right).
\end{align*}
We define:
\[
g_\phi:=\Norm_\phi(f)/(f\circ \phi)\quad\text{in $k(E)^\times$.}
\]
Then we have:
\[
nD_{\phi,x} = \divisor(\Norm_\phi(f)) - \divisor(f\circ \phi) = \divisor(g_\phi) 
\quad\text{on~$E$.} 
\]
This means that $nt^J_\phi(x)$ in $G_x$ is the element $g_\phi(x)/g_\phi(2x)$ of~$k^\times$. By the
construction of~$U$, the divisor of $f$ has support disjoint from that of $g_\phi$ and of
$\phi^*\divisor(f)$ and $\phi_*\divisor(f)$, and Weil reciprocity gives us:
\begin{align*}
\left(\frac{g_\phi(x)}{g_\phi(2x)}\right)^n & = g_\phi(\divisor(f)) =
f(\divisor(g_\phi)) = f(\divisor(\Norm_\phi(f)) - \divisor(f\circ \phi))\\
& = \frac{f(\divisor(\Norm_\phi(f)))}{f(\divisor(f\circ \phi))} = 
\frac{f(\phi_*\divisor(f))}{(f\circ \phi)(\divisor(f))} = 
\frac{f(\phi_*\divisor(f))}{f(\phi_*\divisor(f))} = 1.
\end{align*}
So, indeed $n^2 t^J_\phi(x)=0$ in~$G_x$. Let us also
prove the equality  $nt^J_\phi(x) = e_n(\phi(x),x)$. We have
\[
\text{$\lambda(x) = [(x)-(2x)]$ in $E^\vee$}\,,\quad 
\text{$n((x)-(2x)) =\divisor(f)$ in $\Div(E)$,}
\]
and 
\[
  \lambda(\phi(x)) = [\phi_*(x)-\phi_*(2x)] \quad\text{in $E^\vee$,}
\]
and
\[
n(\phi_*(x)-\phi_*(2x)) = \divisor(\Norm_\phi(f)) \quad \text{in $\Div(E)$}\,.
\]
So, by the description above of the Weil pairing, 
\begin{align*}
e_n(\phi(x),x) & = \frac{(\Norm_\phi(f))((x) - (2x))}{f(\phi_*((x) - (2x)))} = g_\phi((x) - (2x)) \\
& = \frac{g_\phi(x)}{g_\phi(2x)} = nt^J_\phi \quad\text{in $k^\times$.}
\end{align*}

We prove part~3. Let $n$ be a positive odd integer that is prime to
$\deg(\alpha)$, invertible in~$k$, and that divides none among
$\deg(2(\phi-1))$, $\deg(2\phi-1)$ and $\deg(\phi-2)$.  To prove that
there is a $x$ in $U$ such that the order of $x$ is $n$ and the order
of $t^J_\phi(x)$ is~$n^2$, it is sufficient to show that there is an
$x$ in $U$ of order~$n$ such that $e_n(\phi(x),x)$ is of order~$n$. As
$n$ does not divide $\deg(2(\phi-1))$, $\deg(2\phi-1)$, and
$\deg(\phi-2)$, each $x$ in $E$ of order $n$ is in~$U$.

Let now $p$ be a prime number dividing~$n$. Then $p$ is odd, and $p$
is invertible in~$k$, hence $E[p]$ is of dimension two as
$\F_p$-vector space, with the symmetric bilinear form
\[
E[p]\times E[p] \lto k^\times,\quad (x,y)\mapsto e_p(\alpha(x),y).
\]
As $p$ does not divide $\deg(\alpha)$, this form is
perfect. Therefore, there is an $x_p$ in $E[p]$ such that
$e_p(\alpha(x_p),x_p)$ is of order~$p$. Then $e_p(\phi(x_p),x_p)$ is
also of order~$p$, as $e_p(\alpha(x_p),x_p)=e_p(\phi(x_p),x_p)^2$. Let
$n_p$ be the exponent of $p$ in the factorisation of~$n$, and
$x'_p\in E$ such that $x_p=p^{n_p-1}x'_p$, then $x'_p$ is in $E[n]$,
and the order of $e_n(\phi(x'_p),x'_p)$ is~$p^{n_p}$.

Taking for $x$ the sum of the $x'_p$ for $p$ dividing~$n$ gives an $x$
as desired. We have now finished the proof of part~3.

We prove part~4. The two families of extensions of $E$ by $\Gm$ are
fibrewise isomorphic by construction, hence there is a unique
isomorphism of extensions between them as $\Hom(E,\Gm)$ is
trivial. The sections $t_\phi$ and $t^J_\phi$ lie above the graph of
$\alpha\colon E\to E$. We will show that $t^J_\phi$ extends from $U$
to $E$, and that $t_\phi(0)=t^J_\phi(0)$. Then there is a unique $c\in
k^\times$ such that $t^J_\phi=ct_\phi$, and the $c$ equals $1$ because
of the values at~$0$.

We show that $t^J_\phi$ extends from $U$ to $E$ by viewing as
explained above, for $T\to U$, $\Pic(C_T)$ as the group of isomorphism
classes of $(\calL,\sigma)$, with $\calL$ an invertible $\calO$-module
on $E_T$ and $\sigma\colon \Delta^*\calL\to (2\Delta)^*\calL$ an
isomorphism of $\calO$-modules on~$T$. This description extends as
such to all $T\to E$, hence gives us an extension over all of~$E$ of
the extension $G$ of $E_U$ by~${\Gm}_U$. Now we show that
$t^J_\phi$ extends over~$E$. It suffices to take $T=E$, and show that
the divisor $\Delta^*(D_\phi)-(2\Delta)^*(D_\phi)$ on $E$ is principal,
and that the restriction $D_{\phi,0}$ of $D_\phi$ to $E\times\{0\}$ is
principal. 

Definition~(\ref{eq:def_D}) shows that $D_{\phi,0}$ is zero, as
divisor on~$E$. We claim that also
$\Delta^*(D_\phi)-(2\Delta)^*(D_\phi)$ is zero, as divisor on~$E$. We
give the computation. Let $R$ be any $k$-algebra. Then
\begin{align*}
(\Delta^*(\phi_*(\Delta)))(R) & = \{x\in E(R) : \phi(x)=x\}\\
(\Delta^*(\phi_*(2\Delta)))(R) & = \{x\in E(R) : 2\phi(x)=x\}\\
(\Delta^*(\phi^*(\Delta)))(R) & = \{x\in E(R) : \phi(x)=x\}\\
(\Delta^*(\phi^*(2\Delta)))(R) & = \{x\in E(R) : \phi(x)=2x\}\,.
\end{align*}
and 
\begin{align*}
((2\Delta)^*(\phi_*(\Delta)))(R) & = \{x\in E(R) : \phi(x)=2x\}\\
((2\Delta)^*(\phi_*(2\Delta)))(R) & = \{x\in E(R) : 2\phi(x)=2x\}\\
((2\Delta)^*(\phi^*(\Delta)))(R) & = \{x\in E(R) : 2\phi(x)=x\}\\
((2\Delta)^*(\phi^*(2\Delta)))(R) & = \{x\in E(R) : 2\phi(x)=2x\}\,.
\end{align*}
A little bit of bookkeeping shows that the balance is zero.
\end{proof}

\paragraph{Acknowledgements}
\addcontentsline{toc}{section}{Acknowledgements}
We thank Robin de Jong for remarks, corrections and suggestions.  We
also thank the referees of the paper for their comments and
suggestions to improve our text.

\end{document}